\documentclass[11pt]{article}
\usepackage[english,ruled,vlined]{algorithm2e}
\usepackage[noend]{algpseudocode}

\usepackage{pgfplots}
\pgfplotsset{compat=newest}
\usetikzlibrary{shapes}
\usetikzlibrary{plotmarks}
\usepackage{capt-of} 
\usepackage{listings}

\usepackage{amssymb}
\usepackage{amsfonts}
\usepackage{amsmath}
\usepackage{amsthm}
\usepackage{mathrsfs}
\usepackage{graphicx}
\usepackage{amsbsy}
\usepackage{xcolor}
\usepackage[colorlinks=true,linkcolor=mygreen,urlcolor  = blue]{hyperref}

\usepackage{tikz}
\usepackage[normalem]{ulem}
\usepackage{float}
\usepackage{dsfont}

\newtheorem{theorem}{Theorem}[section]
\newtheorem{lemma}[theorem]{Lemma}

\newtheorem{proposition}[theorem]{Proposition}
\newtheorem{property}[theorem]{Property}
\newtheorem{remark}[theorem]{Remark}

\def\beq{\begin{equation}\displaystyle}
\def\eeq{\end{equation}}
\def\bel{\begin{equation} \displaystyle \begin{array}{l} }
\def\eel{\end{array} \end{equation} }
\def\bell{\begin{equation} \displaystyle \begin{array}{ll}  }
\def\eell{\end{array} \end{equation} }

\def\bea{\begin{eqnarray}}
\def\eea{\end{eqnarray} }
\def\bean{\begin{eqnarray*}}
\def\eean{\end{eqnarray*} }
\catcode`@=11
\catcode`@=12

\def\NN{\mathbb{N}}

\def\RR{\mathbb{R}}

\def\ds{\displaystyle}

\def\bs{\bigskip}

\def\eps{\varepsilon}
\def\bar#1{{\overline #1}}

\def\pa{\partial}

\def\1e{\mathds{1}}

\definecolor{mygreen}{rgb}{0,0.7,0}

\begin{document}

\title{Optimal control strategies for the sterile mosquitoes technique}
\author{Luis Almeida
\footnote{Sorbonne Universit{\'e}, CNRS, Universit\'{e} de Paris, Inria, Laboratoire Jacques-Louis Lions UMR7598, F-75005 Paris, France ({\tt luis.almeida@sorbonne-universite.fr})} \and
Michel Duprez\footnote{Inria,  \'equipe MIMESIS, Universit\'e de Strasbourg, Icube, CNRS UMR 7357, Strasbourg, France ({\tt michel.duprez@inria.fr}).}
\and Yannick Privat\footnote{IRMA, Universit\'e de Strasbourg, CNRS UMR 7501, Inria, 7 rue Ren\'e Descartes, 67084 Strasbourg, France ({\tt yannick.privat@unistra.fr}).} 
\and Nicolas Vauchelet\footnote{Laboratoire Analyse, G\'eom\'etrie et Applications CNRS UMR 7539, Universit\'e Sorbonne Paris Nord, Villetaneuse, France ({\tt vauchelet@math.univ-paris13.fr}).}
}

\maketitle

\begin{abstract}
 Mosquitoes are responsible for the transmission of many diseases such as dengue fever, zika or chigungunya. One way to control the spread of these diseases is to use the sterile insect technique (SIT), which consists in a massive release of sterilized male mosquitoes. This strategy aims at reducing the total population over time, and has the advantage being specific to the targeted species, unlike the use of pesticides.
In this article, we study the optimal release strategies in order to maximize the efficiency of this technique.
We consider simplified models that describe the dynamics of eggs, males, females and sterile males in order to optimize the release protocol. We determine in a precise way optimal strategies, which allows us to tackle numerically the underlying optimization problem in a very simple way. We also present some numerical simulations to illustrate our results.
\end{abstract}





\bs


\section{Introduction}

The sterile insect technique (SIT) consists in massively releasing sterilized males in the area where one wishes to reduce the development of certain insects (mosquitoes in this case). Since the released sterile males mate with females, the number of offspring is then reduced, and the size of the insect population diminishes. This strategy was first studied by R. Bushland and E. Knipling and applied successfully in the early 1950's by nearly eradicating screw-worm fly in North America.
Since then, this technique has been considered for different pests and disease vectors \cite{Barclay,SIT}.

Among such vectors, mosquitoes (including Aedes mosquitoes) are responsible for the transmission to humans of many diseases for which there is currently no efficient vaccine nor treatment.
Thus, the sterile insect technique and the closely related incompatible insect technique are very promising tools to control the spread of such diseases by reducing the size of the vector population (there are cases where these techniques have been successfully used to drastically reduce mosquito populations in some isolated regions, e.g. \cite{Bossinetal, Zheng}).

In order to study the efficiency of this technique and to optimize it, mathematical modeling is of great use.
For instance, in \cite{Anguelov,Dumont1,Dumont2,rollingcarpet}, the authors propose mathematical models to study the dynamics of the mosquito population when releasing sterile males. Recently in \cite{bossin}, the authors propose and analyze a differential system modeling the mosquito population dynamics.
Their model is based on experimental observations and  is constructed by assuming that there is a strong Allee effect in the insect population dynamics. A similar model, without strong Allee effect, is investigated in \cite{Anguelov2020}.
Control theory also allows to study the feasibility of controlling the population thanks to the sterile insect technique, and it has been studied in several works, see e.g. \cite{bliman2,bliman3,Aronna}.
Using such mathematical models, authors are able to compare the impact of different strategies in releasing sterile mosquitoes (see e.g. \cite{Cai,LiYuan} and \cite{Dumont2,Huang} where periodic impulsive releases are considered).

In order to find the best possible release protocol, optimal control theory may be used. In \cite{Esteva}, optimal control methods are applied to the rate of introduction of sterile mosquitoes.
An approach developed in \cite{Thome} attempts to control both breeding rates and the rate of introduction of sterile mosquitoes.
In \cite{Fister}, the influence of habitat modification is also considered.
Finally, existence and numerical simulations of the solution to an optimal control problem for the SIT has been proposed in \cite{bliman4}.

In this paper, we study how to optimize the release protocol in order to minimize certain cost functionals, such as the number of mosquitoes.
Starting with the mathematical model presented in \cite{bossin} without Allee effect, we investigate some optimal control problems and focus on obtaining a precise description of the optimal control. To do so, we consider a simplified version of the mathematical model and we perform a complete study of the optimizers.
In particular, in our main result, we describe precisely the optimal release function to minimize the number of sterile males needed to reach a given size of the population of mosquitoes.
Our theoretical results are illustrated with some numerical simulations.
We also provide some extensions to certain related optimization problems in order to illustrate the robustness of our approach.

The outline of this paper is as follow.
In Section \ref{sec:model}, we introduce the mathematical model we will adopt for the sterile insect technique, and describe some useful qualitative properties related to stability issues. For the sake of readability, all the proofs will be postponed to Appendix~\ref{sec:mathpropdynsys}.
Section~\ref{sec:opt} is devoted to the introduction of the problems modeling the search of optimal release protocols and the statement of the main theoretical results providing a precise description of optimal strategies. We then derive a simple algorithm to compute them numerically and provide illustrating simulations.
The proofs of the main theoretical results are postponed to Section~\ref{sec:proof main results}.
Finally, some comments on other possible approaches are gathered in Section~\ref{sec:otherpb}.

\section{Mathematical modelling}\label{sec:model}

\subsection{Mosquito life cycle}

The life cycle of a mosquito (male or female) consists of several stages and takes place successively in two distinct environments:
it includes an aquatic phase (egg, larva, pupa) and an aerial phase (adult).
A few days after mating, a female mosquito may lay a few dozen eggs, possibly spread over several breeding sites.
Once laid, the eggs of some species can withstand hostile environments (including adverse weather conditions) for up to several months before hatching. This characteristic contributes to the adaptability of mosquitoes and has enabled them to colonize temperate regions.
After stimulation (e.g. rainfall), the eggs hatch to give birth to larvae that develop in the water and reach the pupal state. This larval phase can last from a few days to a few weeks. Then, the insect undergoes its metamorphosis. The pupa (also called {\it nymph}) remains in the aquatic state for 1 to 3 days and then becomes an adult mosquito (or {\it imago}): it is the emergence and the beginning of the aerial phase. The lifespan of an adult mosquito is estimated to be of a few weeks.

In many species, egg laying is only possible after a blood meal, i.e. the female must bite a vertebrate before each egg laying. This behavior, called hematophagy, can be exploited by infectious agents (bacteria, viruses or parasites) to spread, alternately from a vertebrate host (humans, for what we are interested in here) to an arthropod host (here, the mosquito).

Based on these observations, a compartmental model has been introduced in \cite{bossin} to model the life cycle of mosquitoes when releasing sterile mosquitoes. 
In what follows, we will both deal with the full and a simplified version of \cite{bossin}. The reason for studying such a simplified model is twofold: on the one hand, the simplified model can be considered relevant from a biological point of view within certain limits. On the other hand, the study of such a ``prototype" model can be considered as a first step towards the development of robust control methodologies with a wider application.

To this aim, we will denote by $u(\cdot)$ a control function standing for a sterile male release function (in other words the rate of sterile male mosquitoes release at each time) and by
\begin{itemize}
\item $M_s(t)$, the sterilized adult males at time $t$;
\item $F(t)$, the adult females that have been fertilized at time $t$.
\end{itemize}
The system we will use for describing the behavior of the mosquito population under the action of the control $u(\cdot)$ reads
\begin{equation}
  \left\{
  \begin{aligned}
&\frac{dF}{dt} = f(F,M_s),  \\
&\frac{dM_s}{dt} = u - \delta_s M_s,
  \end{aligned}
  \right.
  \label{eq:primal2}
  \tag{\mbox{$\mathcal{S}_1$}}
\end{equation}
where $f:\RR^2\to \RR$ denotes the nonlinear function 
\begin{equation}\label{eq:f1}
  f(F,M_s) =  \frac{\nu(1-\nu)\beta_E^2 \nu_E^2 F^2}{\big(\frac{\beta_E F}{K} + \nu_E+\delta_E \big) \big((1-\nu)\nu_E \beta_E F+\delta_M \gamma_s M_s(\frac{\beta_E F}{K} + \nu_E+\delta_E)\big)} - \delta_F F,
\end{equation}
with the following parameter choices:
\begin{itemize}
\item $\beta_E>0$ is the oviposition rate;
\item $\delta_E, \delta_M, \delta_F, \delta_s >0$ are the death rates for eggs, adult males, females, and sterile males respectively;
\item $\nu_E>0$ is the hatching rate for eggs;
\item $\nu\in (0,1)$ the probability that a pupa gives rise to a female, and $(1-\nu)$ is therefore the probability to give rise to a male;
\item 
 $K>0$ is the environmental capacity for eggs. It can be interpreted as the maximum density of eggs that females can lay in breeding sites;
\item $\gamma_s>0$ accounts for the fact that females may have a preference for fertile males. Then, the probability that a female mates with a fertile male is $\frac{M}{M+\gamma_s M_s}$.
\end{itemize}

In the following section, we explain and comment on the choice of this system.

\subsection{Derivation of the simplified model and presentation of the original one}
The choice of \eqref{eq:primal2} as model is inspired by \cite{bossin}. 
To explain how it has been derived, let us present the more involved model we have considered. 
Let us introduce:
\begin{itemize}
\item $E(t)$, the mosquito density in aquatic phase at time $t$;
\item $M(t)$, the adult male density at time $t$;
\item $M_s(t)$, the sterilized adult male density at time $t$;
\item $F(t)$, the density of adult females that has been fertilized at time $t$.
\end{itemize}
Then, the dynamics of the mosquito population is driven by the following dynamical system:
\begin{equation}
  \left\{
    \begin{aligned}
&\frac{dE}{dt}  = \beta_E F \left(1-\frac{E}{K}\right) - \big( \nu_E + \delta_E \big) E,  \\
&\frac{dM}{dt} = (1-\nu)\nu_E E - \delta_M M,  \\
&\frac{dF}{dt} = 
\nu\nu_E E \frac{M}{M+\gamma_s M_s} - \delta_F F, \\
&\frac{dM_s}{dt}  = u - \delta_s M_s.
\end{aligned}
  \right.
  \label{eq:S1}\tag{\mbox{$\mathcal{S}_2$}}
\end{equation}
Regarding this latter model, the main difference with the one in \cite{bossin} is the absence of an exponential term in the equation on $F$ to introduce an Allee effect. This effect reflects the fact that, when the population density is very low, it can be difficult to find a partner to mate. This term is important when considering a small population size. Here, since we are focusing on large populations that we want to reduce in size, we will neglect this term.

Assuming that the time dynamics of the mosquitoes in aquatic phase and the adult males compartments are fast leads to assume that the equations on $E(\cdot)$ and $M(\cdot)$ are at equilibrium. We refer for instance to \cite{MBE} for additional explanations on the justification for these asymptotics.
Hence, we get the following equalities 
\begin{equation*}
E  = \frac{\beta_E F}{\frac{\beta_E F}{K} + \nu_E+\delta_E}\quad \mbox{ and }\quad M  = \frac{(1-\nu)\nu_E}{\delta_M} E \ .
\end{equation*}
Plugging such expressions into \eqref{eq:S1} allows us to obtain \eqref{eq:primal2}.


We conclude this paragraph by numerically comparing the full model \eqref{eq:S1} and the simplified one \eqref{eq:primal2} that we will aim to control.
We consider the numerical values taken from \cite[Table 3]{bossin} and recalled in Table \ref{tab:steril} below.

\begin{table}[H]
\centering
\begin{tabular}{|c|c|c|c|c|}
\hline
\textit{Parameter}&\textit{Name}&\textit{Value interval}&\textit{Chosen value}&\textit{{Unit}}\\\hline
 $\beta_E$&Effective fecundity&7.46--14.85&10  &{Day$^{-1}$}\\\hline
 $\gamma_s$&\begin{tabular}{c}Mating competitiveness\\of sterilizing males\end{tabular}&0--1& 1  &{-}\\ \hline
 $\nu_E$&Hatching parameter&0.005--0.25&
&{Day$^{-1}$}\\\hline
 $\delta_E$&\begin{tabular}{c}Mosquitoes in aquatic phase\\ death rate\end{tabular}&0.023 - 0.046&0.03&{Day$^{-1}$}\\\hline
 $\delta_F$&Female death rate&0.033 - 0.046&0.04&{Day$^{-1}$}\\\hline
  $\delta_M$&Males death rate&0.077 - 0.139&0.1&{Day$^{-1}$}\\\hline
 $\delta_s$&{Sterilized} male  death rate&& 0.12 &{Day$^{-1}$}\\\hline
    $\nu$&Probability of emergence&& 0.49 &{-}\\\hline
\end{tabular}
\caption{Value intervals of the parameters for systems \eqref{eq:S1} and \eqref{eq:primal2} (see \cite{bossin})}
\label{tab:steril}
\end{table}

The numerical results are shown in Fig.~\ref{fig_compar}. In these simulations, the time of the experiment is assumed to be $T=70$ days, and we choose two different release functions $u$: 
$$
u({t})=15\,000\text{ (left),}\quad \text{and}\quad u({t})=20\,000 \sum_{k=0}^6\mathds{1}_{[10k,10k+1]}{(t)}\text{ (right)}.
$$
{In what follows, for a subset $A$ of a given set $E$, the notation $\mathds{1}_{A}$ stands for the characteristic function of $A$. Namely, for each $x\in E$, $\mathds{1}_{A}(x)$ is equal to $1$ if $x\in A$ and $0$ otherwise.}
The dynamics for $F$ (female compartment) is represented for both systems \eqref{eq:S1} (blue, continuous line) and \eqref{eq:primal2} (red, dashed line).
We observe that both problems are very close, which indicates that the dynamics of fertilized females in system \eqref{eq:S1} may be approximated by the one in system \eqref{eq:primal2}.

A mathematical element of this observation lies in the fact that the equilibria {of} Systems \eqref{eq:S1} and \eqref{eq:primal2}  coincide. When dealing with optimal control properties, we will also numerically observe in Section~\ref{sec:num} that this simplification does not affect the optimal strategies in a strong way.

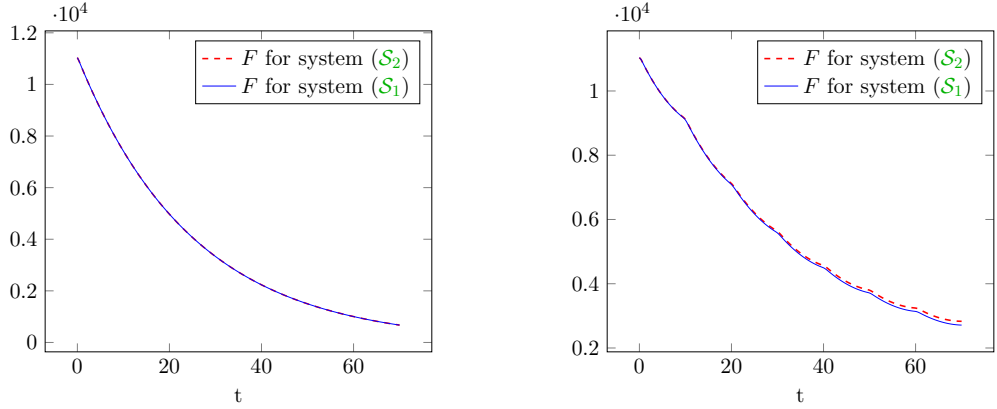
\begin{figure}
~\hfill\begin{tikzpicture}[thick,scale=0.75, every node/.style={scale=1.0}] \begin{axis}[xlabel=t,
,legend pos=north east, legend columns=1]
 \addplot[color=red,dashed,thick]coordinates { 
(0.0,11037.970588235292)
(0.35175879396984927,10909.290434944363)
(0.7035175879396985,10756.875924850963)
(1.0552763819095479,10606.587612598749)
(1.407035175879397,10458.397751343813)
(1.7587939698492463,10312.277633219826)
(2.1105527638190957,10168.198612321403)
(2.462311557788945,10026.132315992802)
(2.814070351758794,9886.050709479076)
(3.1658291457286434,9747.926118958656)
(3.5175879396984926,9611.731239529568)
(3.869346733668342,9477.439136841405)
(4.2211055276381915,9345.02324573944)
(4.572864321608041,9214.45736739309)
(4.92462311557789,9085.715665614813)
(5.276381909547739,8958.77266273354)
(5.628140703517588,8833.603235221815)
(5.9798994974874375,8710.182609190855)
(6.331658291457287,8588.48635582177)
(6.683417085427136,8468.490386775014)
(7.035175879396985,8350.170949604766)
(7.386934673366834,8233.50462319551)
(7.738693467336684,8118.468313232208)
(8.090452261306533,8005.039247711659)
(8.442211055276383,7893.194972500098)
(8.793969849246231,7782.913346940439)
(9.145728643216081,7674.172539511392)
(9.49748743718593,7566.951023539877)
(9.84924623115578,7461.2275729676)
(10.201005025125628,7356.9812581722435)
(10.552763819095478,7254.191441843421)
(10.904522613065327,7152.837774913361)
(11.256281407035177,7052.900192542092)
(11.608040201005027,6954.358910156806)
(11.959798994974875,6857.194419544988)
(12.311557788944725,6761.387485000832)
(12.663316582914574,6666.9191395244225)
(13.015075376884424,6573.7706810731115)
(13.366834170854272,6481.92366886452)
(13.718592964824122,6391.359919730549)
(14.07035175879397,6302.061504521794)
(14.42211055276382,6214.010744561726)
(14.773869346733669,6127.190208150028)
(15.125628140703519,6041.582707114444)
(15.477386934673367,5957.171293410505)
(15.829145728643217,5873.939255768521)
(16.180904522613066,5791.870116387195)
(16.532663316582916,5710.947627673237)
(16.884422110552766,5631.155769026377)
(17.236180904522616,5552.478743669139)
(17.587939698492463,5474.900975520782)
(17.939698492462313,5398.407106114809)
(18.291457286432163,5322.981991559442)
(18.643216080402013,5248.610699540469)
(18.99497487437186,5175.278506365899)
(19.34673366834171,5102.970894051829)
(19.69849246231156,5031.67354744897)
(20.05025125628141,4961.372351409256)
(20.402010050251256,4892.053387992007)
(20.753768844221106,4823.702933709061)
(21.105527638190956,4756.307456808372)
(21.457286432160807,4689.8536145955195)
(21.809045226130653,4624.328250792603)
(22.160804020100503,4559.718392934018)
(22.512562814070353,4496.011249798581)
(22.864321608040203,4433.194208877514)
(23.216080402010054,4371.2548338777815)
(23.5678391959799,4310.18086226029)
(23.91959798994975,4249.9602028124655)
(24.2713567839196,4190.580933254727)
(24.62311557788945,4132.031297880392)
(24.974874371859297,4074.299705228537)
(25.326633165829147,4017.374725789361)
(25.678391959798997,3961.2450897416006)
(26.030150753768847,3905.8996847215412)
(26.381909547738694,3851.327553623194)
(26.733668341708544,3797.517892429196)
(27.085427135678394,3744.460048072008)
(27.437185929648244,3692.14351632499)
(27.78894472361809,3640.5579397229326)
(28.14070351758794,3589.6931055116356)
(28.49246231155779,3539.5389436261294)
(28.84422110552764,3490.0855246971378)
(29.19597989949749,3441.323058085389)
(29.547738693467338,3393.2418899433874)
(29.899497487437188,3345.8325013042563)
(30.251256281407038,3299.085506197285)
(30.603015075376888,3252.9916497897975)
(30.954773869346734,3207.541806554978)
(31.306532663316585,3162.7269784652935)
(31.658291457286435,3118.5382932111556)
(32.01005025125628,3074.967002444463)
(32.36180904522613,3032.004480046688)
(32.71356783919598,2989.642220421153)
(33.06532663316583,2947.8718368091677)
(33.41708542713568,2906.6850596296886)
(33.76884422110553,2866.073734842176)
(34.12060301507538,2826.029822332321)
(34.47236180904523,2786.545394320328)
(34.824120603015075,2747.6126337914284)
(35.175879396984925,2709.223832948328)
(35.527638190954775,2671.3713916852703)
(35.879396984924625,2634.0478160834205)
(36.231155778894475,2597.245716927269)
(36.582914572864325,2560.957808241763)
(36.934673366834176,2525.176905849875)
(37.286432160804026,2489.8959259503245)
(37.63819095477387,2455.107883715168)
(37.98994974874372,2420.805891906986)
(38.34170854271357,2386.983159515384)
(38.69346733668342,2353.6329904125496)
(39.04522613065327,2320.7487820275855)
(39.39698492462312,2288.3240240393716)
(39.74874371859297,2256.3522970876816)
(40.10050251256282,2224.8272715023127)
(40.45226130653267,2193.742706049968)
(40.80402010050251,2163.092446698647)
(41.15577889447236,2132.8704253992996)
(41.50753768844221,2103.070658884501)
(41.85929648241206,2073.6872474839115)
(42.21105527638191,2044.7143739562857)
(42.56281407035176,2016.1463023378003)
(42.91457286432161,1987.9773768064713)
(43.26633165829146,1960.2020205624378)
(43.618090452261306,1932.8147347238896)
(43.969849246231156,1905.81009723842)
(44.321608040201006,1879.1827618095886)
(44.67336683417086,1852.9274568384812)
(45.02512562814071,1827.0389843800588)
(45.37688442211056,1801.5122191140858)
(45.72864321608041,1776.3421073304382)
(46.08040201005026,1751.5236659285865)
(46.43216080402011,1727.051981431059)
(46.78391959798995,1702.9222090106866)
(47.1356783919598,1679.1295715314404)
(47.48743718592965,1655.6693586026686)
(47.8391959798995,1632.5369256465483)
(48.19095477386935,1609.7276929785653)
(48.5427135678392,1587.2371449008422)
(48.89447236180905,1565.060828808133)
(49.2462311557789,1543.1943543063094)
(49.597989949748744,1521.6333923431616)
(49.949748743718594,1500.373674351345)
(50.301507537688444,1479.4109914033002)
(50.653266331658294,1458.7411933779797)
(51.005025125628144,1438.3601881392185)
(51.356783919597994,1418.2639407255815)
(51.708542713567844,1398.4484725515326)
(52.060301507537694,1378.909860619763)
(52.412060301507545,1359.6442367445254)
(52.76381909547739,1340.6477867858187)
(53.11557788944724,1321.916749894273)
(53.46733668341709,1303.4474177665847)
(53.81909547738694,1285.2361339113545)
(54.17085427135679,1267.2792929251814)
(54.52261306532664,1249.5733397788733)
(54.87437185929649,1232.1147691136282)
(55.22613065326634,1214.90012454705)
(55.57788944723618,1197.92599798886)
(55.92964824120603,1181.1890289661692)
(56.28140703517588,1164.6859039581775)
(56.63316582914573,1148.4133557401683)
(56.98492462311558,1132.3681627366686)
(57.33668341708543,1116.5471483836468)
(57.68844221105528,1100.9471804996201)
(58.04020100502513,1085.5651706655515)
(58.39195979899498,1070.3980736134074)
(58.743718592964825,1055.4428866232604)
(59.095477386934675,1040.6966489288143)
(59.447236180904525,1026.1564411312359)
(59.798994974874375,1011.8193846211752)
(60.150753768844226,997.6826410088626)
(60.502512562814076,983.7434115621666)
(60.854271356783926,969.9989366525047)
(61.206030150753776,956.4464952084946)
(61.55778894472362,943.0834041772397)
(61.90954773869347,929.9070179931398)
(62.26130653266332,916.9147280541256)
(62.61306532663317,904.1039622052091)
(62.96482412060302,891.4721842292508)
(63.31658291457287,879.0168933448416)
(63.66834170854272,866.7356237111993)
(64.02010050251256,854.6259439399839)
(64.37185929648241,842.6854566139325)
(64.72361809045226,830.9117978122202)
(65.07537688442211,819.3026366424526)
(65.42713567839196,807.8556747791966)
(65.77889447236181,796.5686460089595)
(66.13065326633166,785.4393157815248)
(66.48241206030151,774.4654807675571)
(66.83417085427136,763.6449684223876)
(67.18592964824121,752.9756365558947)
(67.53768844221106,742.4553729083937)
(67.88944723618091,732.0820947324523)
(68.24120603015076,721.853748380549)
(68.59296482412061,711.7683088984916)
(68.94472361809046,701.8237796245181)
(69.2964824120603,692.0181917939974)
(69.64824120603015,682.3496041496535)
(70.0,672.8161025572356)
 };

  \addplot[color=blue]coordinates { 

(0.0,11037.970588235292)
(0.35175879396984927,10909.290421346795)
(0.7035175879396985,10756.875888319533)
(1.0552763819095479,10606.587552686364)
(1.407035175879397,10458.397668275142)
(1.7587939698492463,10312.27752737035)
(2.1105527638190957,10168.198484095201)
(2.462311557788945,10026.132165782778)
(2.814070351758794,9886.050537652784)
(3.1658291457286434,9747.92592585377)
(3.5175879396984926,9611.731025453366)
(3.869346733668342,9477.438902071968)
(4.2211055276381915,9345.022990527554)
(4.572864321608041,9214.457091964357)
(4.92462311557789,9085.715370171762)
(5.276381909547739,8958.772347457618)
(5.628140703517588,8833.60290027522)
(5.9798994974874375,8710.182254718204)
(6.331658291457287,8588.485981951588)
(6.683417085427136,8468.489993621073)
(7.035175879396985,8350.170537267275)
(7.386934673366834,8233.504191762171)
(7.738693467336684,8118.467862779169)
(8.090452261306533,8005.038778304349)
(8.442211055276383,7893.1944841939785)
(8.793969849246231,7782.912839781678)
(9.145728643216081,7674.172013537462)
(9.49748743718593,7566.950478780096)
(9.84924623115578,7461.227009443619)
(10.201005025125628,7356.9806758984805)
(10.552763819095478,7254.19084082746)
(10.904522613065327,7152.837155156314)
(11.256281407035177,7052.899554038922)
(11.608040201005027,6954.35825289663)
(11.959798994974875,6857.193743511348)
(12.311557788944725,6761.3867901719495)
(12.663316582914574,6666.9184258734285)
(13.015075376884424,6573.769948568262)
(13.366834170854272,6481.922917469397)
(13.718592964824122,6391.359149404247)
(14.07035175879397,6302.060715219089)
(14.42211055276382,6214.009936233244)
(14.773869346733669,6127.1893807423985)
(15.125628140703519,6041.581860570442)
(15.477386934673367,5957.170427669192)
(15.829145728643217,5873.938370765376)
(16.180904522613066,5791.869212054237)
(16.532663316582916,5710.946703939152)
(16.884422110552766,5631.154825816628)
(17.236180904522616,5552.477780906077)
(17.587939698492463,5474.8999931237595)
(17.939698492462313,5398.4061040002825)
(18.291457286432163,5322.980969641072)
(18.643216080402013,5248.609657729224)
(18.99497487437186,5175.277444570151)
(19.34673366834171,5102.96981217745)
(19.69849246231156,5031.672445399426)
(20.05025125628141,4961.371229085704)
(20.402010050251256,4892.052245293383)
(20.753768844221106,4823.701770532175)
(21.105527638190956,4756.306273047998)
(21.457286432160807,4689.852410144483)
(21.809045226130653,4624.327025541875)
(22.160804020100503,4559.717146772801)
(22.512562814070353,4496.009982614403)
(22.864321608040203,4433.192920556315)
(23.216080402010054,4371.253524304003)
(23.5678391959799,4310.179531316967)
(23.91959798994975,4249.958850381315)
(24.2713567839196,4190.579559216239)
(24.62311557788945,4132.029902113921)
(24.974874371859297,4074.298287612391)
(25.326633165829147,4017.373286200897)
(25.678391959798997,3961.243628057314)
(26.030150753768847,3905.89820081716)
(26.381909547738694,3851.3260473737732)
(26.733668341708544,3797.5163637092096)
(27.085427135678394,3744.4584967554465)
(27.437185929648244,3692.141942285455)
(27.78894472361809,3640.5563428337327)
(28.14070351758794,3589.6914856458843)
(28.49246231155779,3539.537300656843)
(28.84422110552764,3490.0838584973335)
(29.19597989949749,3441.3213685281858)
(29.547738693467338,3393.240176902104)
(29.899497487437188,3345.8307646525145)
(30.251256281407038,3299.08374580911)
(30.603015075376888,3252.9898655397205)
(30.954773869346734,3207.539998318139)
(31.306532663316585,3162.725146117547)
(31.658291457286435,3118.536436629172)
(32.01005025125628,3074.965121505837)
(32.36180904522613,3032.0025746300425)
(32.71356783919598,2989.640290406246)
(33.06532663316583,2947.869882076999)
(33.41708542713568,2906.6830800626085)
(33.76884422110553,2866.071730323992)
(34.12060301507538,2826.0277927484076)
(34.47236180904523,2786.543339557734)
(34.824120603015075,2747.610553738987)
(35.175879396984925,2709.221727496767)
(35.527638190954775,2671.369260727321)
(35.879396984924625,2634.045659513928)
(36.231155778894475,2597.2435346433035)
(36.582914572864325,2560.955600142729)
(36.934673366834176,2525.1746718376226)
(37.286432160804026,2489.893665929259)
(37.63819095477387,2455.105597592361)
(37.98994974874372,2420.803579592285)
(38.34170854271357,2386.980820921524)
(38.69346733668342,2353.6306254552605)
(39.04522613065327,2320.7463906257044)
(39.39698492462312,2288.321606114949)
(39.74874371859297,2256.3498525660925)
(40.10050251256282,2224.824800312363)
(40.45226130653267,2193.740208124003)
(40.80402010050251,2163.0899219726575)
(41.15577889447236,2132.867873813028)
(41.50753768844221,2103.068080381546)
(41.85929648241206,2073.6846420118336)
(42.21105527638191,2044.711741466709)
(42.56281407035176,2016.1436427865153)
(42.91457286432161,1987.9746901535354)
(43.26633165829146,1960.1993067722753)
(43.618090452261306,1932.8119937653896)
(43.969849246231156,1905.8073290850332)
(44.321608040201006,1879.1799664394227)
(44.67336683417086,1852.924634234395)
(45.02512562814071,1827.0361345297533)
(45.37688442211056,1801.5093420101955)
(45.72864321608041,1776.3392029706188)
(46.08040201005026,1751.5207343156017)
(46.43216080402011,1727.0490225728643)
(46.78391959798995,1702.9192229205125)
(47.1356783919598,1679.1265582278718)
(47.48743718592965,1655.6663181097229)
(47.8391959798995,1632.5338579937506)
(48.19095477386935,1609.724598201022)
(48.5427135678392,1587.234023039311)
(48.89447236180905,1565.057679909091)
(49.2462311557789,1543.1911784220176)
(49.597989949748744,1521.6301895317288)
(49.949748743718594,1500.370444676786)
(50.301507537688444,1479.4077349355937)
(50.653266331658294,1458.7379101931213)
(51.005025125628144,1438.3568783192713)
(51.356783919597994,1418.2606043587245)
(51.708542713567844,1398.4451097321048)
(52.060301507537694,1378.9064714483043)
(52.412060301507545,1359.6408213278146)
(52.76381909547739,1340.6443452369083)
(53.11557788944724,1321.9132823325203)
(53.46733668341709,1303.4439243176796)
(53.81909547738694,1285.2326147073434)
(54.17085427135679,1267.2757481044882)
(54.52261306532664,1249.569769486316)
(54.87437185929649,1232.1111735004329)
(55.22613065326634,1214.89650377086)
(55.57788944723618,1197.9223522137422)
(55.92964824120603,1181.1853583626164)
(56.28140703517588,1164.6822087031069)
(56.63316582914573,1148.4096360169162)
(56.98492462311558,1132.3644187349814)
(57.33668341708543,1116.5433802996677)
(57.68844221105528,1100.943388535873)
(58.04020100502513,1085.56135503092)
(58.39195979899498,1070.39423452311)
(58.743718592964825,1055.4390242988234)
(59.095477386934675,1040.6927635980385)
(59.447236180904525,1026.152533028161)
(59.798994974874375,1011.8154539860403)
(60.150753768844226,997.6786880880618)
(60.502512562814076,983.7394366082021)
(60.854271356783926,969.9949399239355)
(61.206030150753776,956.4424769698813)
(61.55778894472362,943.0793646990859)
(61.90954773869347,929.9029575518296)
(62.26130653266332,916.9106469318574)
(62.61306532663317,904.099860689926)
(62.96482412060302,891.4680626145673)
(63.31658291457287,879.0127519299665)
(63.66834170854272,866.7314628008558)
(64.02010050251256,854.6217638443253)
(64.37185929648241,842.6812576484555)
(64.72361809045226,830.9075802976743)
(65.07537688442211,819.2984009047461)
(65.42713567839196,807.8514211493003)
(65.77889447236181,796.5643748228064)
(66.13065326633166,785.435027379907)
(66.48241206030151,774.4611754960198)
(66.83417085427136,763.6406466311205)
(67.18592964824121,752.9712985996201)
(67.53768844221106,742.4510191462524)
(67.88944723618091,732.0777255278863)
(68.24120603015076,721.8493641011819)
(68.59296482412061,711.763909916007)
(68.94472361809046,701.8193663145352)
(69.2964824120603,692.0137645359439)
(69.64824120603015,682.3451633266375)
(70.0,672.8116485559144)

  };

 \legend{$F$ for system \eqref{eq:S1},$F$ for system \eqref{eq:primal2}}
\end{axis} 
\end{tikzpicture}
\hfill
\begin{tikzpicture}[thick,scale=0.75, every node/.style={scale=1.0}] \begin{axis}[xlabel=t,
,legend pos=north east, legend columns=1]
 \addplot[color=red,dashed,thick]coordinates { 

(0.0,11037.970588235292)
(0.35175879396984927,10981.42970667781)
(0.7035175879396985,10884.049470197475)
(1.0552763819095479,10784.740777900717)
(1.407035175879397,10688.27649620718)
(1.7587939698492463,10594.629507314316)
(2.1105527638190957,10503.772660321265)
(2.462311557788945,10415.678621650726)
(2.814070351758794,10330.319750667493)
(3.1658291457286434,10247.667996287753)
(3.5175879396984926,10167.694811274383)
(3.869346733668342,10090.37108161432)
(4.2211055276381915,10015.66706891898)
(4.572864321608041,9943.552364212002)
(4.92462311557789,9873.995851796562)
(5.276381909547739,9806.965682147784)
(5.628140703517588,9742.429252970249)
(5.9798994974874375,9680.35319770883)
(6.331658291457287,9620.703380912983)
(6.683417085427136,9563.444899937957)
(7.035175879396985,9508.542092527416)
(7.386934673366834,9455.958549865654)
(7.738693467336684,9405.657134717925)
(8.090452261306533,9357.600004297617)
(8.442211055276383,9311.748637511722)
(8.793969849246231,9268.063866243296)
(9.145728643216081,9226.505910333233)
(9.49748743718593,9187.034415924869)
(9.84924623115578,9139.949878620393)
(10.201005025125628,9063.796754518049)
(10.552763819095478,8972.981426756825)
(10.904522613065327,8876.502569194463)
(11.256281407035177,8781.04031955737)
(11.608040201005027,8687.872978444502)
(11.959798994974875,8596.990618302103)
(12.311557788944725,8508.383577753613)
(12.663316582914574,8422.04240180535)
(13.015075376884424,8337.95778454249)
(13.366834170854272,8256.120513804102)
(13.718592964824122,8176.521417449968)
(14.07035175879397,8099.151310934421)
(14.42211055276382,8024.000945987808)
(14.773869346733669,7951.060960277493)
(15.125628140703519,7880.32182798018)
(15.477386934673367,7811.773811247664)
(15.829145728643217,7745.406912590435)
(16.180904522613066,7681.210828239165)
(16.532663316582916,7619.174902573708)
(16.884422110552766,7559.288083733857)
(17.236180904522616,7501.538880545943)
(17.587939698492463,7445.915320915185)
(17.939698492462313,7392.40491184558)
(18.291457286432163,7340.9946012575165)
(18.643216080402013,7291.67074177826)
(18.99497487437186,7244.419056682337)
(19.34673366834171,7199.224608157657)
(19.69849246231156,7156.071768069197)
(20.05025125628141,7110.384261650449)
(20.402010050251256,7050.106948047539)
(20.753768844221106,6980.834640133081)
(21.105527638190956,6907.710682370374)
(21.457286432160807,6835.474405132497)
(21.809045226130653,6765.077383684092)
(22.160804020100503,6696.514800687825)
(22.512562814070353,6629.782097967354)
(22.864321608040203,6564.8749395054565)
(23.216080402010054,6501.789174406024)
(23.5678391959799,6440.520799705355)
(23.91959798994975,6381.065922951831)
(24.2713567839196,6323.420724503035)
(24.62311557788945,6267.581419516195)
(24.974874371859297,6213.544219632166)
(25.326633165829147,6161.305294375144)
(25.678391959798997,6110.860732310473)
(26.030150753768847,6062.206502021233)
(26.381909547738694,6015.338412981086)
(26.733668341708544,5970.252076416158)
(27.085427135678394,5926.942866262568)
(27.437185929648244,5885.40588033863)
(27.78894472361809,5845.635901861815)
(28.14070351758794,5807.627361450087)
(28.49246231155779,5771.374299755365)
(28.84422110552764,5736.870330883455)
(29.19597989949749,5704.108606759832)
(29.547738693467338,5673.0817826041275)
(29.899497487437188,5640.25874148931)
(30.251256281407038,5595.9238376796075)
(30.603015075376888,5544.060664360769)
(30.954773869346734,5488.81056803166)
(31.306532663316585,5434.2499885569105)
(31.658291457286435,5381.210764829298)
(32.01005025125628,5329.691234330157)
(32.36180904522613,5279.689974076232)
(32.71356783919598,5231.2057728528225)
(33.06532663316583,5184.237602951033)
(33.41708542713568,5138.7845913658075)
(33.76884422110553,5094.84599042711)
(34.12060301507538,5052.421147851313)
(34.47236180904523,5011.509476213672)
(34.824120603015075,4972.110421855931)
(35.175879396984925,4934.223433255602)
(35.527638190954775,4897.847928895367)
(35.879396984924625,4862.98326468246)
(36.231155778894475,4829.628700978692)
(36.582914572864325,4797.783369312088)
(36.934673366834176,4767.4462388507945)
(37.286432160804026,4738.616082729029)
(37.63819095477387,4711.291444323283)
(37.98994974874372,4685.4706035847175)
(38.34170854271357,4661.151543540644)
(38.69346733668342,4638.331917084154)
(39.04522613065327,4617.009014176196)
(39.39698492462312,4597.179729588706)
(39.74874371859297,4578.840531320743)
(40.10050251256282,4558.655667947782)
(40.45226130653267,4527.443611785661)
(40.80402010050251,4488.976230019768)
(41.15577889447236,4447.167006412837)
(41.50753768844221,4405.863217522833)
(41.85929648241206,4365.851428856281)
(42.21105527638191,4327.132147429116)
(42.56281407035176,4289.706095771862)
(42.91457286432161,4253.5741895324445)
(43.26633165829146,4218.7375144962225)
(43.618090452261306,4185.197303001401)
(43.969849246231156,4152.954909737318)
(44.321608040201006,4122.011786922113)
(44.67336683417086,4092.369458865022)
(45.02512562814071,4064.0294959270086)
(45.37688442211056,4036.9934879016987)
(45.72864321608041,4011.263016846626)
(46.08040201005026,3986.839629402609)
(46.43216080402011,3963.7248086466593)
(46.78391959798995,3941.919945531159)
(47.1356783919598,3921.4263099690747)
(47.48743718592965,3902.245021631713)
(47.8391959798995,3884.3770205318942)
(48.19095477386935,3867.823037471383)
(48.5427135678392,3852.583564436962)
(48.89447236180905,3838.6588250345717)
(49.2462311557789,3826.04874505546)
(49.597989949748744,3814.7529232722454)
(49.949748743718594,3801.7590816506213)
(50.301507537688444,3778.703716537648)
(50.653266331658294,3748.9367430753564)
(51.005025125628144,3716.0191816473116)
(51.356783919597994,3683.4965017166587)
(51.708542713567844,3652.096680623242)
(52.060301507537694,3621.8216227360467)
(52.412060301507545,3592.673437973221)
(52.76381909547739,3564.654423973171)
(53.11557788944724,3537.7670476736785)
(53.46733668341709,3512.013926283699)
(53.81909547738694,3487.397807638393)
(54.17085427135679,3463.921549933696)
(54.52261306532664,3441.5881008424612)
(54.87437185929649,3420.4004760197827)
(55.22613065326634,3400.361737010692)
(55.57788944723618,3381.4749685788825)
(55.92964824120603,3363.7432554805337)
(56.28140703517588,3347.1696587126357)
(56.63316582914573,3331.7571912704257)
(56.98492462311558,3317.5087934536637)
(57.33668341708543,3304.4273077664297)
(57.68844221105528,3292.5154534599233)
(58.04020100502513,3281.7758007723437)
(58.39195979899498,3272.2107449243244)
(58.743718592964825,3263.8224799325244)
(59.095477386934675,3256.612972307861)
(59.447236180904525,3250.5839347084393)
(59.798994974874375,3245.7367996204835)
(60.150753768844226,3239.0990433753777)
(60.502512562814076,3222.4941868170654)
(60.854271356783926,3199.302100242278)
(61.206030150753776,3173.002898091624)
(61.55778894472362,3147.003698705442)
(61.90954773869347,3122.002604099809)
(62.26130653266332,3098.0025131355364)
(62.61306532663317,3075.006520304058)
(62.96482412060302,3053.0179010064867)
(63.31658291457287,3032.0400962835643)
(63.66834170854272,3012.076696982773)
(64.02010050251256,2993.13142735293)
(64.37185929648241,2975.2081280605667)
(64.72361809045226,2958.3107386263164)
(65.07537688442211,2942.443279283459)
(65.42713567839196,2927.609832264606)
(65.77889447236181,2913.8145225263543)
(66.13065326633166,2901.0614979255315)
(66.48241206030151,2889.354908864406)
(66.83417085427136,2878.698887425936)
(67.18592964824121,2869.09752602379)
(67.53768844221106,2860.554855595445)
(67.88944723618091,2853.0748233701806)
(68.24120603015076,2846.6612702472007)
(68.59296482412061,2841.3179078224407)
(68.94472361809046,2837.0482951058275)
(69.2964824120603,2833.8558149738587)
(69.64824120603015,2831.7436504053485)
(70.0,2830.714760551018)

 };

  \addplot[color=blue]coordinates { 

(0.0,11037.970588235292)
(0.35175879396984927,10981.4117845632)
(0.7035175879396985,10883.965352346777)
(1.0552763819095479,10784.537132908952)
(1.407035175879397,10687.901122900576)
(1.7587939698492463,10594.034807379625)
(2.1105527638190957,10502.915111802511)
(2.462311557788945,10414.518364986288)
(2.814070351758794,10328.820264438786)
(3.1658291457286434,10245.795844273947)
(3.5175879396984926,10165.419445920665)
(3.869346733668342,10087.664691822283)
(4.2211055276381915,10012.5044623107)
(4.572864321608041,9939.910875823596)
(4.92462311557789,9869.855272616029)
(5.276381909547739,9802.30820209842)
(5.628140703517588,9737.239413912008)
(5.9798994974874375,9674.617852830414)
(6.331658291457287,9614.411657552151)
(6.683417085427136,9556.588163424038)
(7.035175879396985,9501.113909109687)
(7.386934673366834,9447.954647191005)
(7.738693467336684,9397.075358663964)
(8.090452261306533,9348.44027126337)
(8.442211055276383,9302.012881525054)
(8.793969849246231,9257.75598046828)
(9.145728643216081,9215.631682756393)
(9.49748743718593,9175.601459170237)
(9.84924623115578,9127.983689921088)
(10.201005025125628,9051.410669281428)
(10.552763819095478,8960.280729436397)
(10.904522613065327,8863.537830586249)
(11.256281407035177,8767.789517919802)
(11.608040201005027,8674.29890461485)
(11.959798994974875,8583.05693522229)
(12.311557788944725,8494.054761862913)
(12.663316582914574,8407.283710816982)
(13.015075376884424,8322.735248018647)
(13.366834170854272,8240.4009435008)
(13.718592964824122,8160.272434844849)
(14.07035175879397,8082.341389699099)
(14.42211055276382,8006.599467438879)
(14.773869346733669,7933.038280051364)
(15.125628140703519,7861.649352337903)
(15.477386934673367,7792.42408153679)
(15.829145728643217,7725.353696479536)
(16.180904522613066,7660.429216403808)
(16.532663316582916,7597.64140955628)
(16.884422110552766,7536.980751728439)
(17.236180904522616,7478.4373848779705)
(17.587939698492463,7422.001075997516)
(17.939698492462313,7367.661176401184)
(18.291457286432163,7315.406581607235)
(18.643216080402013,7265.225692002526)
(18.99497487437186,7217.1063744805915)
(19.34673366834171,7171.035925250479)
(19.69849246231156,7127.001034017406)
(20.05025125628141,7080.486743293333)
(20.402010050251256,7019.623986427204)
(20.753768844221106,6949.95413114773)
(21.105527638190956,6876.529469435019)
(21.457286432160807,6803.97631634962)
(21.809045226130653,6733.2222657933735)
(22.160804020100503,6664.262961414183)
(22.512562814070353,6597.0943290689875)
(22.864321608040203,6531.712554383496)
(23.216080402010054,6468.114059216706)
(23.5678391959799,6406.295477017047)
(23.91959798994975,6346.25362706079)
(24.2713567839196,6287.985487567548)
(24.62311557788945,6231.488167692454)
(24.974874371859297,6176.75887839975)
(25.326633165829147,6123.794902228354)
(25.678391959798997,6072.593561966283)
(26.030150753768847,6023.152188257806)
(26.381909547738694,5975.468086174803)
(26.733668341708544,5929.538500792102)
(27.085427135678394,5885.360581815535)
(27.437185929648244,5842.931347321079)
(27.78894472361809,5802.247646673823)
(28.14070351758794,5763.306122706439)
(28.49246231155779,5726.103173248478)
(28.84422110552764,5690.634912109926)
(29.19597989949749,5656.897129635137)
(29.547738693467338,5624.885252956221)
(29.899497487437188,5591.148945258702)
(30.251256281407038,5546.21208283114)
(30.603015075376888,5493.986115360289)
(30.954773869346734,5438.497486277135)
(31.306532663316585,5383.680866281804)
(31.658291457286435,5330.337973923166)
(32.01005025125628,5278.467619329036)
(32.36180904522613,5228.068914777674)
(32.71356783919598,5179.14125956983)
(33.06532663316583,5131.6843239277705)
(33.41708542713568,5085.698031877845)
(33.76884422110553,5041.182543073129)
(34.12060301507538,4998.1382335129365)
(34.47236180904523,4956.56567511676)
(34.824120603015075,4916.465614111444)
(35.175879396984925,4877.838948192283)
(35.527638190954775,4840.686702421331)
(35.879396984924625,4805.0100038295695)
(36.231155778894475,4770.810054693933)
(36.582914572864325,4738.08810446544)
(36.934673366834176,4706.845420331139)
(37.286432160804026,4677.083256400173)
(37.63819095477387,4648.802821513147)
(37.98994974874372,4622.005245684263)
(38.34170854271357,4596.691545197299)
(38.69346733668342,4572.862586389616)
(39.04522613065327,4550.519048172884)
(39.39698492462312,4529.661383355146)
(39.74874371859297,4510.28977884611)
(40.10050251256282,4489.168447988005)
(40.45226130653267,4457.404802390711)
(40.80402010050251,4418.674260953565)
(41.15577889447236,4376.749387134356)
(41.50753768844221,4335.308627116551)
(41.85929648241206,4295.102927168018)
(42.21105527638191,4256.133444368902)
(42.56281407035176,4218.40165105972)
(42.91457286432161,4181.909325208857)
(43.26633165829146,4146.658539901697)
(43.618090452261306,4112.651651884226)
(43.969849246231156,4079.891289090762)
(44.321608040201006,4048.3803370826035)
(44.67336683417086,4018.1219243219466)
(45.02512562814071,3989.119406203418)
(45.37688442211056,3961.376347764324)
(45.72864321608041,3934.896504994179)
(46.08040201005026,3909.6838046645385)
(46.43216080402011,3885.742322601705)
(46.78391959798995,3863.0762603277267)
(47.1356783919598,3841.6899199993723)
(47.48743718592965,3821.5876775806873)
(47.8391959798995,3802.7739541924097)
(48.19095477386935,3785.253185591152)
(48.5427135678392,3769.0297897429496)
(48.89447236180905,3754.1081324696506)
(49.2462311557789,3740.492491162767)
(49.597989949748744,3728.187016577861)
(49.949748743718594,3714.294166265621)
(50.301507537688444,3690.7758024690106)
(50.653266331658294,3660.866664104428)
(51.005025125628144,3627.9687717386946)
(51.356783919597994,3595.440109013194)
(51.708542713567844,3563.969196614255)
(52.060301507537694,3533.5586901237157)
(52.412060301507545,3504.211582326154)
(52.76381909547739,3475.9311986191474)
(53.11557788944724,3448.7211916654956)
(53.46733668341709,3422.5855352029384)
(53.81909547738694,3397.528516920983)
(54.17085427135679,3373.5547303086273)
(54.52261306532664,3350.6690653712326)
(54.87437185929649,3328.8766981096164)
(55.22613065326634,3308.1830786498404)
(55.57788944723618,3288.593917908263)
(55.92964824120603,3270.1151726734643)
(56.28140703517588,3252.753028984791)
(56.63316582914573,3236.513883686784)
(56.98492462311558,3221.404324039828)
(57.33668341708543,3207.4311052702938)
(57.68844221105528,3194.6011259484317)
(58.04020100502513,3182.92140108961)
(58.39195979899498,3172.3990328843606)
(58.743718592964825,3163.0411789753707)
(59.095477386934675,3154.8550182151816)
(59.447236180904525,3147.8477138571293)
(59.798994974874375,3142.0263741540466)
(60.150753768844226,3134.54875880784)
(60.502512562814076,3117.5974511810077)
(60.854271356783926,3094.4115414126422)
(61.206030150753776,3068.292266431139)
(61.55778894472362,3042.442499408868)
(61.90954773869347,3017.5176495065216)
(62.26130653266332,2993.5215319425747)
(62.61306532663317,2970.4583201623973)
(62.96482412060302,2948.3325452634517)
(63.31658291457287,2927.149094752598)
(63.66834170854272,2906.9132105365784)
(64.02010050251256,2887.630486038465)
(64.37185929648241,2869.306862324518)
(64.72361809045226,2851.9486231176843)
(65.07537688442211,2835.5623885659807)
(65.42713567839196,2820.155107626486)
(65.77889447236181,2805.734048918741)
(66.13065326633166,2792.3067898953163)
(66.48241206030151,2779.8812041723713)
(66.83417085427136,2768.4654468594777)
(67.18592964824121,2758.067937726122)
(67.53768844221106,2748.6973420424442)
(67.88944723618091,2740.362548934238)
(68.24120603015076,2733.0726470973937)
(68.59296482412061,2726.8368977251034)
(68.94472361809046,2721.6647045126547)
(69.2964824120603,2717.565580619753)
(69.64824120603015,2714.549112489382)
(70.0,2712.6249204453875)

  };
 \legend{$F$ for system \eqref{eq:S1},$F$ for system \eqref{eq:primal2}}
\end{axis} 
\end{tikzpicture}\hfill~
  \caption{Comparisons of the numerical solutions $F$ (in continous blue line for System \eqref{eq:S1} and in dashed red line for System~\eqref{eq:primal2} or equation \ref{init}). The initial conditions correspond to the ``persistence'' equilibrium (see propositions \ref{prop:4 eq steady} and \ref{prop:2 eq steady}) Left: $u(\cdot)=15\,000$. Right : the releases occur every $10$ days with an intensity of $20\,000$ mosquitoes in one day, i.e. $u({t})=20\,000 \sum_{k=0}^6\mathds{1}_{[10k,10k+1]}{(t)}$.}\label{fig_compar}
\end{figure}

%
%

%

%
 
\subsection{Mathematical properties of the dynamical systems}

This section is devoted to establishing stability properties for equilibria of Systems~\eqref{eq:S1} and \eqref{eq:primal2} in the absence of control, in order to qualitatively understand their behavior whenever initial data are chosen close to equilibria. These results can be considered as preliminary tools before investigating optimal control properties for these problems.

Recall that both systems share the same steady states. We will moreover show that they enjoy the same stability properties. For the sake of readability, all proofs are postponed to Appendix~\ref{sec:mathpropdynsys}.

In what follows, we will make the following assumption, in accordance with the numerical values gathered in Table~\ref{tab:steril}:
\begin{equation}\tag{\mbox{$\mathcal{H}$}}\label{eq:cond prop}
 \delta_s>\delta_M\quad\text{and}\quad {\mathcal{R}_0:=\frac{\nu\beta_E \nu_E}{\delta_F(\nu_E+\delta_E)}>1},
\end{equation}
{where $\mathcal{R}_0$ denotes the so-called basic offspring number (number of adult females produced by one adult female during her lifespan).  }

\begin{proposition}[Stability properties for System~\eqref{eq:S1}]
  \label{prop:4 eq steady}
  Let us assume that \eqref{eq:cond prop} holds. 
\begin{enumerate}
\item If $u(\cdot)=0$, then, System~\eqref{eq:S1} has two equilibria:
\begin{itemize}
\item the ``extinction'' equilibrium $(E_1^*,M_1^*,F_1^*,M_{s1}^*)=(0,0,0,0)$, which
  is linearly unstable\footnote{{Meaning that at least one eigenvalue of the Jacobian matrix of the system has a positive real part}};
\item the ``persistence'' equilibrium $(E_2^*,M_2^*,F_2^*,M_{s2}^*)=\left(\bar E,\bar M,\bar F,0\right)$, where
\begin{equation}\label{eq:equi}
 \bar E=K \left(1-{\frac{1}{\mathcal{R}_0}}\right),  \qquad
\bar  M=\frac{(1-\nu)\nu_E}{\delta_M}\bar  E,\qquad 
\bar  F=\frac{\nu\nu_E}{\delta_F}\bar  E,
\end{equation}
%
which is {locally asymptotically stable (LAS)}.
\end{itemize}
\item If the control function $u$ is assumed to be non-negative, then the corresponding solution $(E,M,$ $F,M_s)$ to System~\eqref{eq:S1} enjoys the following stability property:
\begin{equation*}
\left\{\begin{array}{l}
E(0)\in(0,\bar E]\\
M(0)\in(0,\bar M]\\
F(0)\in(0,\bar F]\\
M_{s}(0)\geqslant 0
\end{array}\right.
\Longrightarrow
\left\{\begin{array}{l}
E(t)\in(0,\bar E]\\
M(t)\in(0,\bar M]\\
F(t)\in(0,\bar F]\\
M_{s}(t)\geqslant 0
\end{array}\right.
\mbox{ for all }t\geqslant 0.
\end{equation*}
Finally, let $U^*$ be defined by
  \begin{equation}\label{eq:U*}
    U^* := {\mathcal{R}_0}\frac{K (1-\nu) \nu_E \delta_s}{4\gamma_s\delta_M}\left(1-{\frac{1}{\mathcal{R}_0}}\right)^2
  \end{equation}
  and let $\bar{U}$ denote any positive number such that $\bar{U} > U^*$.
  If $u(\cdot)$ denotes the constant control function almost everywhere equal to $\bar{U}$ for all $t\geqslant 0$, then the corresponding solution $(E(t),M(t),F(t))$  to System~\eqref{eq:S1} converges to the extinction equilibrium as $t\to +\infty$.
\end{enumerate}
\end{proposition}

{
\begin{remark}
We verify from the first point of this proposition that $\mathcal{R}_0>1$ implies the population persistence while $\mathcal{R}_0\leqslant 1$ expresses the population extinction.
\end{remark}
}

In the following result, we will again use the notations introduced in Prop.~\ref{prop:4 eq steady} above.

\begin{proposition}[Stability properties for System~\eqref{eq:primal2}]
  \label{prop:2 eq steady}
  Let us assume that \eqref{eq:cond prop} holds.
 \begin{enumerate} 
 \item If $u(\cdot)=0$, System~\eqref{eq:primal2} has two equilibria:
\begin{itemize}
\item the ``extinction'' equilibrium $(F_1^*,M_{s1}^*)=(0,0)$, which
  is unstable if $\delta_s > \delta_F$.
\item the ``persistence'' equilibrium $(F_2^*,M_{s2}^*)=\left(\bar F,0\right)$,
 is {locally asymptotically stable (LAS)}.
\end{itemize}
\item If the control function $u$ is assumed to be non-negative, then the corresponding solution $(F,M_s)$ to System~\eqref{eq:primal2} enjoys the following stability property:
\begin{equation*}
\left\{\begin{array}{l}
F(0)\in(0,\bar F]\\
M_{s}(0)\geqslant 0
\end{array}\right.
\Longrightarrow
\left\{\begin{array}{l}
F(t)\in(0,\bar F]\\
M_{s}(t)\geqslant 0
\end{array}\right.
\mbox{ for all }t\geqslant 0.
\end{equation*}
 If $u(\cdot)$ denotes the constant control function almost everywhere equal to $\bar{U} > U^*$ (defined by \eqref{eq:U*}), then $F(t)$ goes to $0$ as $t\to +\infty$.
\end{enumerate}
\end{proposition}


As a consequence of Propositions~\ref{prop:4 eq steady} and \ref{prop:2 eq steady}, it follows that if the horizon of time of control $T$ is large enough, by releasing a sufficient amount of mosquitoes, it is possible to make the wild population of mosquitoes be as small as desired at time $T$.
In the next section, we will investigate the issue of minimizing the number of released mosquitoes in order to reach a given size of the wild population.
In Section~\ref{sec:otherpb}, we will also comment on other relevant minimization problems related to the sterile insect technique.

\section{Optimal control problems}\label{sec:opt}
In what follows, we will consider an initial state corresponding to the beginning of the experiment, in which there are no sterile mosquitoes. Hence, the mosquito population is at the ``persistence'' equilibrium at the beginning of the experiment, meaning that one chooses
\begin{equation}\label{init} \begin{array}{c}
\quad F(0) = \bar{F}, \quad M_s(0)=0,  \mbox{ for System \eqref{eq:S1}, to which we should add } \\
E(0)=\bar E,\quad M(0)=\bar M, \mbox{ for System~\eqref{eq:S1}}
\end{array}
\end{equation}
where $(\bar{E},\bar{M},\bar{F})$ are defined by \eqref{eq:equi}.
Our aim is to determine the optimal time distribution of the releases, such that the number of mosquitoes used to reach a desired size of the population is as small as possible.

\subsection{Statement of the problems and main results}
Let us first model the optimization of the release procedure. Given the duration of the experiment, we aim at minimizing the amount of mosquitoes required to reduce the size of the wild population of mosquitoes to a given value:
\begin{center}
\textit{What should be the time distribution of optimal releases in order to reach a given value of the wild population at the end of the experiment, by using as few sterilized mosquitoes as possible?}
\end{center}
To answer this question, we first define a cost functional that will stand for the objective we are trying to achieve.
We mention that other formulations of minimization problem are also proposed in Section~\ref{sec:otherpb}.

Let $T>0$ be a given horizon of time, $\bar U>0$ be a maximal amount of sterilized mosquitoes, and $\varepsilon>0$ be a given target amount of female mosquitoes. We will investigate the issue above for both Systems~\eqref{eq:primal2} and \eqref{eq:S1}. We model the optimal control problems for the sterile mosquito release strategy as
\begin{equation}\label{eq:opt2}
\boxed{\inf_{u\in \mathcal{U}_{T,\overline{U},\varepsilon}^{(\mathcal{S}_1)}} J(u)},
\tag{\mbox{$\mathcal{P}_{T,\overline{U},\varepsilon}^{(\mathcal{S}_1)}$}}
\end{equation}
and
\begin{equation}\label{eq:opt1}
\boxed{\inf_{u\in \mathcal{U}_{T,\overline{U},\varepsilon}^{(\mathcal{S}_2)}} J(u)},
\tag{\mbox{$\mathcal{P}_{T,\overline{U},\varepsilon}^{(\mathcal{S}_2)}$}}
\end{equation}
where the functional $J$ stands for the total number of released mosquitoes during the time $T$, namely
\begin{equation*}
J(u):=\int_0^Tu(t)dt.
\end{equation*} 
For a given $\varepsilon>0$, we introduce the sets of admissible controls for Systems~\eqref{eq:primal2} and \eqref{eq:S1}, respectively denoted $\mathcal{U}_{T,\overline{U},\varepsilon}^{(\mathcal{S}_1)}$ and $\mathcal{U}_{T,\overline{U},\varepsilon}^{(\mathcal{S}_2)}$ and defined by
\begin{equation}\label{calU2}\begin{array}{c}
{\mathcal U}_{T,\bar{U},\varepsilon}^{(\mathcal{S}_1)} := \Big\{ u\in L^\infty(0,T):0\leqslant u \leqslant \overline{U} \text{ a.e.},F(T)\leqslant \varepsilon\hspace{3cm}\\\hspace{3cm}\mbox{ with }F\mbox{ solution of System~\eqref{eq:primal2}}\Big\}\end{array}
\end{equation}
and
\begin{equation*}\begin{array}{c}
{\mathcal U}_{T,\bar{U},\varepsilon}^{(\mathcal{S}_2)} := \Big\{ u\in L^\infty(0,T):0\leqslant u \leqslant \overline{U} \text{ a.e.},F(T)\leqslant \varepsilon\hspace{3cm}\\\hspace{3cm}\mbox{ with }F\mbox{ solution of System~\eqref{eq:S1}}\Big\}.\end{array}
\end{equation*}

\begin{remark}[Exact null-controllability]
It is notable that there does not exist any control $u\in L^{\infty}(0,T;\RR_+)$  such that the corresponding solution $F$ to System \eqref{eq:primal2} (resp. System~\eqref{eq:S1}) satisfies $F(T)=0$ since one has $F(t)\geqslant F(0)e^{-\delta_Ft}$ for  every $t\in [0,T]$. 
\end{remark}

Let us now state the main results of this article. In what follows, we will use the notation $U^*$ to denote the positive number defined by \eqref{eq:U*}, and the notations  $J_{T,\overline{U},\varepsilon}^{(\mathcal{S}_i)}$, $i=1,2$ to denote the optimal values for Problems~\eqref{eq:opt2} and \eqref{eq:opt1}, namely
\begin{equation}\label{optval}
J_{T,\overline{U},\varepsilon}^{(\mathcal{S}_i)} := \inf_{u\in {\mathcal U}_{T,\overline{U},\varepsilon}^{(\mathcal{S}_i)}} J(u) , \qquad i=1,2.
\end{equation}
\begin{theorem}\label{theo:opt 1}
Let us assume that Condition \eqref{eq:cond prop} holds true.
Let $\eps\in(0,\bar{F})$ and $\overline{U} > U^*$. There exists a minimal time $\overline{T}>0$ such that for all $T\geqslant \overline{T}$, the set ${\mathcal U}_{T,\overline{U},\varepsilon}^{(\mathcal{S}_2)}$ is nonempty and Problem~\eqref{eq:opt1} has a solution $u^*$.
One has $J_{T,\overline{U},\varepsilon}^{(\mathcal{S}_2)} \leqslant \bar{U}\, \bar{T} $ and the mappings $T\in [\bar T,+\infty)\mapsto J_{T,\overline{U},\varepsilon}^{(\mathcal{S}_2)}$ and $ \bar U\in (U^*,+\infty)\mapsto J_{T,\overline{U},\varepsilon}^{(\mathcal{S}_2)}$ are non-increasing.
Furthermore, there exists $t_1\in(0,T)$  such that
 $$u^*=0\quad \mbox{ on }(t_1,T),$$
one has $F(T)=\varepsilon$ and $F\in[\varepsilon,\bar F]$ on $(0,T)$.
\end{theorem}

It is interesting to note that there is no need to release sterile mosquitoes at the end of the experiment.
For the simplified system~\eqref{eq:primal2}, we can obtain a much more precise characterization of optimal controls, which is the purpose of the next result. As a preliminary remark, recall that the function $f$ is defined by \eqref{eq:f1}.
\begin{theorem}\label{theo:opt 2}
Let us assume that Condition \eqref{eq:cond prop} holds true and that $2\delta_s>\delta_F$. Let $\eps\in(0,\bar{F})$, $\overline{U} > U^*$. 
There exists a minimal time $\overline{T}>0$ such that for all $T\geqslant \overline{T}$,
the set ${\mathcal U}_{T,\overline{U},\varepsilon}^{(\mathcal{S}_1)}$ is nonempty and Problem~\eqref{eq:opt2} has a solution $u^*$, characterized as follows:
\begin{enumerate}
\item \textsf{Optimal value: }one has $J_{T,\overline{U},\varepsilon}^{(\mathcal{S}_1)}\leqslant \bar{U}\, \bar{T}$ and the mappings $ T\in [\bar T,+\infty)\mapsto J_{T,\overline{U},\varepsilon}^{(\mathcal{S}_1)}$ and $\bar U\in (U^*,+\infty)\mapsto J_{T,\overline{U},\varepsilon}^{(\mathcal{S}_1)}$ are non-increasing.
\item \textsf{Optimal control: }let $T > \bar{T}$. If $\overline{U} > U^*$ is large enough, there exists $(t_0,t_1)\in[0,T]^2$ with $t_0 \leqslant t_1$ such that
\begin{enumerate}
\item $u^*=0$ on $(0,t_0)$ or $u^*=\overline{U}$ on $(0,t_0)$;
\item $u^*\in(0,\overline{U})$ on $(t_0,t_1)$ with 
\begin{equation}\label{3eq:edo u}
u^*(t) = \left.\left(\frac{\partial^2 f}{\partial M_s^2}\right)^{-1} \left(\frac{\partial f}{\partial M_s}\frac{\partial f}{\partial F}+ \delta_s M_s\frac{\partial^2 f}{\partial M_s^2}-f\frac{\partial^2 f}{\partial M_s\partial F}  \right)\right|_{(F(t),M_s(t))};
\end{equation}
\item $u^*=0$ on $(t_1,T)$.
\end{enumerate}
Furthermore, there exists $T^*(\overline{U})$ (simply denoted $T^*$ when no confusion is possible) such that $T^*>\bar{T}$ and that if $T>T^*$, then the optimizer $u^*$ is unique, with $0<t_0<t_1<T$, and such that $u^*=0$ on $(0,t_0)$. Finally, the mapping $T\in (T^*,+\infty)\mapsto J_{T,\bar{U},\varepsilon}^{(\mathcal{S}_1)}$ is constant.
\item \textsf{Optimal trajectory: }we extend the domain of definition of $F$ to $\RR_+$ by setting $u(\cdot)=0$ on $(T,+\infty)$. One has $F(T)=\varepsilon$, $F'\leqslant 0$ on $(0,T)$ and $F'\geqslant 0$ on $(T,+\infty)$. In particular, $F$ has a unique local minimum at $T$ which is moreover global, equal to $\eps$ and $F'(T)=0$.
\end{enumerate}
\end{theorem}

Let us comment on this result. Under the assumptions of Theorem~\ref{theo:opt 2} (in particular that $T$ and $\bar U$ are large enough), the infinite dimensional control problem \eqref{eq:opt2} can be reduced to a two dimensional one: more precisely, one only needs to determine the two parameters $t_0$ and $t_1$.  
  
 As expected, the horizon of time $T$ fixed for the control to reach a desired number of adult female mosquitoes has a strong influence on the form of the optimal control.
  The longer $T$ is, the smaller is the number of sterilized males needed. However, there is a maximal time $T^*$ above which the number of released mosquitoes needed to reach the desired state is stationary with respect to the horizon of time.

\begin{remark}{{(Stabilization and feedback)}}
Note that our results can be interpreted in terms of stabilization: indeed, it follows from Theorem~\ref{theo:opt 2}, and in particular from the description of the optimal trajectory that, under the assumptions of this theorem, the differential system
$$
\left\{\begin{array}{ll}
\displaystyle\frac{d}{dt}\begin{pmatrix}F\\ M_s\end{pmatrix}  = \begin{pmatrix} f(F,M_s)\\ u - \delta_s M_s \end{pmatrix} & \text{in }[0,+\infty)
\\
\text{with }
u =\frac{\frac{\partial f}{\partial M_s}(F,M_s)\frac{\partial f}{\partial F}(F,M_s)+ \frac{\partial^2 f}{\partial M_s^2}(F,M_s) \delta_s M_s -\frac{\partial^2 f}{\partial M_s\partial F}(F,M_s)f(F,M_s) }{\frac{\partial^2 f}{\partial M_s^2}(F,M_s)} & ,
\end{array}\right.
$$
complemented with the initial conditions $F(0) = \bar{F}$ and $M_s(0)=0$, satisfies
$$
\lim_{t\to +\infty}F(t)=0,
$$
in other words, the control function $u(\cdot)$ above defines a feedback control stabilizing System~\eqref{eq:primal2}.
\end{remark}


\subsection{A dedicated algorithm}\label{sec:algo}
In this section, we introduce an elementary algorithm for computing the (unique) solution to Problem~\eqref{eq:opt2}. The chosen approach mainly rests upon the precise description of the optimizer (whenever $T$ and $\bar U$ are large enough) provided in Theorem~\ref{theo:opt 2}.

Let us describe our approach. With the notations of Theorem~\ref{theo:opt 2}, we assume that $T>T^*$ so that the optimal control is unique. We will take advantage of its particular form, more precisely that there exist $0<t_0<t_1<T$ such that  $u^*=0$ on $(0,t_0)$, $u^*$ is given by \eqref{3eq:edo u} on $(t_0,t_1)$, $u^*=0$ on $(t_1,T)$ and $F(T)=\eps$. 

To this aim, let us introduce for $\tau_1\in [0,T]$ the auxiliary Cauchy system
\begin{equation}
\left\{\begin{array}{l}
\displaystyle\frac{d}{dt}\begin{pmatrix}F\\ M_s\end{pmatrix}  = \begin{pmatrix} f(F,M_s)\\ u_{\tau_1} - \delta_s M_s \end{pmatrix}  \text{\quad in }[0,+\infty)
\\
\text{with }
u_{\tau_1} =\frac{\frac{\partial f}{\partial M_s}(F,M_s)\frac{\partial f}{\partial F}(F,M_s)+ \frac{\partial^2 f}{\partial M_s^2}(F,M_s) \delta_s M_s -\frac{\partial^2 f}{\partial M_s\partial F}(F,M_s)f(F,M_s) }{\frac{\partial^2 f}{\partial M_s^2}(F,M_s)}\mathds{1}_{(0,\tau_1)}  ,
\end{array}\right.
  \label{eq:primal3}
\end{equation}
complemented with the initial conditions $F(0) = \bar{F}$ and $M_s(0)=0$, whose solution, associated to the control function $u_{\tau_1}$, will from now on be denoted $(F^{\tau_1},M_s^{\tau_1})$. Let $T>T^*$.

The uniqueness property for Problem~\eqref{eq:opt2} allows us to get the following relations between the solution to Problem~\eqref{eq:opt2} and the control $u_{\tau_1}$.
\begin{property}\label{prop:num}
Under the assumptions and with the notations of Theorem~\ref{theo:opt 2}, let $\bar U>U^*$ (in particular we consider $T\geqslant T^*(\bar U)$ so that the conclusion (ii) holds true). There exists $\varepsilon_0>0$ such that, if $\varepsilon \in (0,\varepsilon_0)$, then 
\begin{enumerate}
\item for every $\tau_1\in [0,T)$,
there exists a unique $\tau_2(\tau_1)\in (\tau_1,+\infty)$ such that 
$F^{\tau_1}$ is first strictly decreasing on $(0,\tau_2(\tau_1))$, and then
strictly increasing on $(\tau_2(\tau_1),+\infty)$. 
\item the value function
$$
\psi:   \tau_1\in [0,T)\mapsto \min_{t\in [0,\infty]} F^{\tau_1}(t)=F^{\tau_1}(\tau_2(\tau_1))
$$
is decreasing.
\item the optimal control $u^*$ solving Problem~\eqref{eq:opt2} satisfies
\begin{equation}\label{eq:opt2 bis}
\begin{array}{c}
u^*(t)=\left\{\begin{array}{ll}
0&\mbox{ if }t\in(0,T-\tau_2(\tau_1)),\\
{u_{\tau_1}}(t-T+\tau_2(\tau_1))&\mbox{ otherwise}
\end{array}\right.\\
\end{array}
\end{equation}
a.e. on $(0,T)$, where $\tau_1$ denotes the unique solution on $[0,T)$ to the equation $\psi(\tau_1)=\varepsilon$.
\end{enumerate}
\end{property}
%

We now construct an efficient algorithm based on this result to solve Problem~\eqref{eq:opt2}, involving a bisection type method.
{The resulting Algorithm is described in Figure \ref{alg:Figure1}.}

\begin{algorithm}[!htbp]
\captionof{figure}{
{Algorithm by dichotomy to solve Problem~\eqref{eq:opt2}}}
\label{alg:Figure1}
\RestyleAlgo{tworuled}\vspace*{2.5mm}
\SetKwProg{init}{Initialization}{:}{}\vspace*{1.5mm}
\init{}{\vspace{3mm}
Let $n\in\mathbb{N}^*$, $\tau_{1,\mbox{\scriptsize{min}}}=0$ and 
$\tau_{1,\mbox{\scriptsize{max}}}=T$
}
\vspace*{2.5mm}
\SetKwProg{while}{While}{ do}

\while{$i \leq n$}
{
\medskip
\begin{algorithmic}[1]
 \State
 $\tau_{1,\mbox{\scriptsize{test}}}= (\tau_{1,\mbox{\scriptsize{min}}}+\tau_{1,\mbox{\scriptsize{max}}})/2$ 
 \State 
  Solve  \eqref{eq:primal3} on $(0,T)$ for $\tau_1=\tau_{1,\mbox{\scriptsize{test}}}$ and let {$u^{\tau_1}$} be the function  given by \eqref{eq:primal3}.  

  \State \If{$\min_{(0,T)} F^{\tau_1}<\varepsilon$}
  
$\tau_{1,\mbox{\scriptsize{max}}}=\tau_{1,\mbox{\scriptsize{test}}}$

\State   \Else{$\tau_{1,\mbox{\scriptsize{min}}}=\tau_{1,\mbox{\scriptsize{test}}}$}
 
\end{algorithmic}
}
\SetKwProg{init}{End}{:}{}\vspace*{1.5mm}
\init{}{\vspace{3mm}
Let $\tau_1^n=(\tau_{1,\mbox{\scriptsize{min}}}+\tau_{1,\mbox{\scriptsize{max}}})/2
$, $t_1^n=\mbox{argmin}_{(0,T)} F^{\tau_1^n}$ and 
$$u_n(t)=\left\{\begin{array}{ll}
0&\mbox{ if }t\in(0,T-t_1^n),\\
u(t-T+t_1^n)&\mbox{ otherwise},
\end{array}\right.$$
for each $t\in(0,T)$.
}

\label{algo:dicho}
\end{algorithm}

%
%
%

\subsection{Numerical simulations}\label{sec:num}

In order to illustrate our main results (Theorems~\ref{theo:opt 1} and \ref{theo:opt 2}), we provide some numerical simulations of the optimal control problems \eqref{eq:opt2} and \eqref{eq:opt1}.
The codes are available on \texttt{github}\footnote{ 
\url{github.com/michelduprez/Optimal-control-strategies-for-the-sterile-mosquitoes-technique.git}
}.

We use the parameter values provided in Table~\ref{tab:steril}, that come from \cite[Table 1-3]{bossin}.
As in \cite{bossin}, in order to get results relevant for an island
of $74\,ha$ (hectares) with an estimated male population of about $69\,ha^{-1}$, the total number of males at the beginning of the experiment is taken equal to $M^*=69\times 74=5106$.
Regarding System~\eqref{eq:S1}, we assume that $M_s=0$ at the equilibrium, and we then deduce that
\begin{equation*}
\bar E = \frac{\delta_M}{(1-\nu)\nu_E} \bar M,  
\qquad
\bar F=\frac{ 
\nu\nu_E }{\delta_F}\bar E= 
 \frac{\nu\delta_M}{\delta_F(1-\nu)} \bar M 
\end{equation*}
and we thus evaluate the environmental capacity for eggs as
\begin{equation*}
  K=\left(1-\frac{ \delta_F\big( \nu_E + \delta_E \big) }{\beta_E 
  \nu\nu_E}\right)^{-1}\bar E.
\end{equation*}

Regarding System~\eqref{eq:primal2}, we consider the same initial quantity of females $\bar F$
 and get that the environmental capacity $K$ has the same expression as above.


To compute the solutions to Problem~\eqref{eq:opt2} and  \eqref{eq:opt1}, we use the opensource optimization routine \textsc{GEKKO} (see \cite{beal2018gekko}) 
which solves the optimization problem thanks to the \textsc{APOPT} (Advanced Process OPTimizer) library, a software package for solving large-scale optimization problems (see \cite{hedengren2012apopt}).
We will also compare the numerical solution obtained when solving Problem~\eqref{eq:opt2} with the one obtained by Formula~\eqref{eq:opt2 bis}, by applying our algorithm (see Figure~\ref{alg:Figure1}).

Figures \ref{fig:steril1}, \ref{fig:steril2} and \ref{fig:steril3} gather the solutions of the optimal control problems \eqref{eq:opt2}, \eqref{eq:opt1} and the function given by \eqref{eq:opt2 bis} for $T=60, 150, 200$, {$\overline{U}=5000$}, $\nu_E=0.05$ and $\varepsilon=\bar F/4$.

The time interval $(0,T)$ is discretized with $300$ points. 
We do not give the solution to \eqref{eq:opt2 bis}  for $T=60$  in Figure \ref{fig:steril1}, since Algorithm \ref{algo:dicho} cannot be applied in this situation (indeed, $T$ is not large enough and the {form of the corresponding solution is not convenient to apply a bisection procedure}).

 {For these parameter values, the bound $U^*$ given in \eqref{eq:U*} is approximately equal to $9620$. We remark that this bound is not optimal since, as we can see in Figure \ref{fig:steril1}-\ref{fig:steril3}, the optimal control problems admits a solution for $\overline{U}=5000$.}

These simulations allow us to recover the theoretical results of Section~\ref{sec:opt}. {Indeed, the optimal control has the structure described in Theorem~\ref{theo:opt 2}: at the beginning of the time interval, it vanishes (Figure~\ref{fig:steril1}) or it is equal to $\bar U$ (Figures~\ref{fig:steril2}-\ref{fig:steril3}), then we observe a singular part and finally a time interval for which the control vanishes (Figures~\ref{fig:steril1}-\ref{fig:steril3}).}

Moreover, it is interesting to observe that, even if we do not have a proof of this fact, the optimal controls for the different problems look very close. This is confirmed by Table \ref{tab:table1}, where the quantity of sterile insects is computed for different values of $\nu_E$.
{We remark also that $\nu_E$ has not a big influence on the total number of released mosquitoes.}

{We can observe that the length of the ``waiting time'' $T -\tau_2(\tau_1)$ for which $u^* = 0$ is not impacted by the total length of the time interval $T$. Indeed, this ``waiting time'' can be biologically and intuitively interpreted by the fact that the sterilized males continue to act after the last release.}

Regarding Problem~\eqref{eq:opt2}, since for $T$ large enough, the optimal control vanishes on an interval at the beginning of the experiment, the system stays at the equilibrium on this interval. Hence, with the notations of Theorem~\ref{theo:opt 2}, the optimal time  $T_{\mbox{\scriptsize{opt}}}$  to control the system with a singular part is equal to $T-t_0$. 
As illustrated in Table \ref{tab:table2}, this optimal time  $T_{\mbox{\scriptsize{opt}}}$ is not very sensitive to the choice of optimal control problem, namely  \eqref{eq:opt2}, \eqref{eq:opt1} and \eqref{eq:opt2 bis} and also of the value of $\nu_E$.  

In Table \ref{tab:table3}, we {also} provide, for different discretizations of the time interval, the computation time for solving Problems~\eqref{eq:opt2} and  \eqref{eq:opt1} with the open-source optimization routine \textsc{GEKKO} (default parameters) and the one for solving \eqref{eq:opt2 bis} thanks to Algorithm~\ref{algo:dicho} ($n=50$ iterations). 
Here, we have used an \texttt{Intel CORE i5 8th Gen.}
As expected, one can notice that Algorithm~\ref{algo:dicho} yields a much faster resolution than the other approaches.

{We finally give in Figure~\ref{fig:steril long time} the solution of the optimal control problem~\eqref{eq:opt1} for a small $\varepsilon=\bar F/1e4$ and a large final time $T$ thanks to our numerical algorithm (see Figure~\ref{alg:Figure1}) (the numerical resolution thanks to Gekko does not converge). We observe that in this case the control is increasing then decreasing. In other words, we act with large releases at the beginning of the time interval and then with small releases. Even if we do not take into account the Allee effect, we can remark that we do not need to release a lot of sterilized males at the end of the time interval.
}

\begin{figure}[H]
\begin{center}~\hfill
\hfill~
\end{center}
\caption{Solution of the optimal control problems \eqref{eq:opt2}  and \eqref{eq:opt1} 
with  $T=60$, $\overline{U}=5000$, $\nu_E=0.05$,  and $\varepsilon=\bar F/4$.}
\label{fig:steril1}
\end{figure}

\begin{figure}[H]
\begin{center}~\hfill
%
\hfill~

%
\end{center}
\caption{Solution of the optimal control problems \eqref{eq:opt2}, \eqref{eq:opt1}  and \eqref{eq:opt2 bis} 
with  $T=150$, $\overline{U}=5000$, $\nu_E=0.05$,  and $\varepsilon=\bar F/4$.}
\label{fig:steril2}
\end{figure}

\begin{figure}[H]
\begin{center}~\hfill

\hfill~
\end{center}
\caption{Solution of the optimal control problems \eqref{eq:opt2}, \eqref{eq:opt1}  and \eqref{eq:opt2 bis} 
with  $T=200$, $\overline{U}=5000$, $\nu_E=0.05$,  and $\varepsilon=\bar F/4$.}
\label{fig:steril3}
\end{figure}


\begin{table}[H]
\centering

\begin{tabular}{|c|cccc|}
\hline
\textit{$\nu_E$}&0.005&0.01&0.02&0.03\\\hline
$J(u^*)$ for \eqref{eq:opt1}
&1.21e5&1.34e5&1.42e5&1.45e5\\ 
$J(u^*)$ for \eqref{eq:opt2} 
&1.15e5&1.30e5&1.40e5&1.43e5\\ 
$J(u^*)$ for \eqref{eq:opt2 bis}
&1.15e5&1.30e5&1.40e5&1.43e5\\     \hline
\end{tabular}

\medskip

\begin{tabular}{|c|cccc|}
\hline
\textit{$\nu_E$}&0.05&0.1&0.15&0.25\\\hline
$J(u^*)$ for \eqref{eq:opt1}
&1.47e5&1.49e5&1.50e5&1.50e5\\ 
$J(u^*)$ for \eqref{eq:opt2} 
&1.46e5&1.49e5&1.49e5&1.50e5\\ 
$J(u^*)$ for \eqref{eq:opt2 bis}
&1.46e5&1.48e5&1.49e5&1.50e5\\     \hline
\end{tabular}

\caption{$J(u^*)$ for \eqref{eq:opt2}, \eqref{eq:opt1} and \eqref{eq:opt2 bis} with respect to $\nu_E$  for $\varepsilon=\bar F/4$ with $\bar F=11037$. }
\label{tab:table1}
\end{table}

\begin{table}[H]
\centering
\begin{tabular}{|c|cccccccc|}
\hline
\textit{$\nu_E$}&0.005&0.01&0.02&0.03&0.05&0.1&0.15&0.25\\\hline
$T_{\mbox{\scriptsize{opt}}}$ for \eqref{eq:opt1}&109&107&105&105&104&104&104&104 \\ 
$T_{\mbox{\scriptsize{opt}}}$ for \eqref{eq:opt2}&106&105&104&104&104&104&104&104 \\ 
$T_{\mbox{\scriptsize{opt}}}$ for \eqref{eq:opt2 bis}&106&105&104&103&103&103&103&103\\       \hline
\end{tabular}
\caption{Optimal time $T_{\mbox{\scriptsize{opt}}}$ of control for \eqref{eq:opt2}, \eqref{eq:opt1} and \eqref{eq:opt2 bis} with respect to $\nu_E$  for  $\varepsilon=\bar F/4$ with $\bar F=11037$. }
\label{tab:table2}
\end{table}

\begin{table}[H]
\centering
\begin{tabular}{|c|ccccccc|}
\hline
\textit{Time discr.}&50&100&200&400&800&1600&3200\\\hline
C. t. for \eqref{eq:opt1}&2.64&8.11&2.56e1&1.86e2&3.83e4&X&X\\ 
C. t. for \eqref{eq:opt2} 
&2.05&6.00&2.94e1&1.13e2&7.37e2&3.53e4&X\\ 
C. t. for \eqref{eq:opt2 bis}
&3.58e-1&6.74e-1&1.44&2.73&5.45&11.7&2.24e1\\      \hline
\end{tabular}
\caption{Time of computation (sec.) for \eqref{eq:opt2}, \eqref{eq:opt1} and \eqref{eq:opt2 bis} with respect to the time discretization for $\nu_E=0.05$, $\varepsilon=\bar F/4$ with $\bar F=11037$. }
\label{tab:table3}
\end{table}

{
\begin{figure}[H]
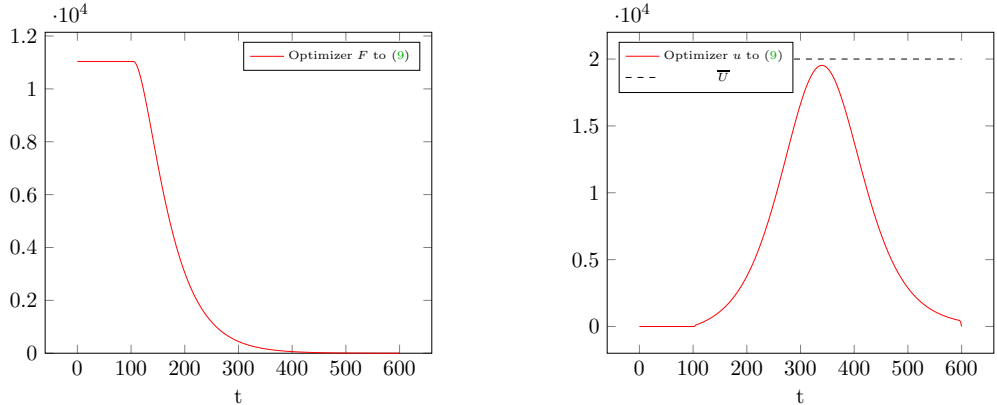

\begin{center}~\hfill
%
\hfill~
\end{center}
\caption{{Solution of the optimal control problems  \eqref{eq:opt2 bis} 
with  $T=600$, $\overline{U}=20000$, $\nu_E=0.05$,  and $\varepsilon=\bar F/1e4$.}}
\label{fig:steril long time}
\end{figure}
 }

\color{black}

\section{Proofs of the main results}\label{sec:proof main results}

\subsection{Proof of Theorem \ref{theo:opt 1}}
In the whole proof, we will denote by $(f_E(E,F)$, $f_M(E,M)$, $f_F(E,M,F,M_s))^\top$ the right-hand side of the differential subsystem of System~\eqref{eq:S1} satisfied by $(E,M,F)^\top$, namely
$$
f_E(E,F) = \beta_E F \left(1-\frac{E}{K}\right) - \big( \nu_E + \delta_E \big) E,$$
$$f_M(E,M) = (1-\nu)\nu_E E - \delta_M M$$and$$
 f_F(E,M,F,M_s)=\nu\nu_E E \frac{M}{M+\gamma_s M_s} - \delta_F F.
$$

We first point out that system \eqref{eq:S1} enjoys a monotonicity property with respect to the control $u$.
\begin{lemma}\label{monotonu}
  Let $u_1, u_2\in L^{\infty}(0,T;\RR_+)$ be such that  $u_1\geqslant u_2$. Let us assume that \eqref{init} holds.
  
  Then, the associated solutions $(E_1,M_1,F_1,M_{s1})$, $(E_2,M_2,F_2,M_{s2})$ to System \eqref{eq:S1} respectively associated to $u_1$ and $u_2$ satisfy $(E_1,M_1,F_1)\leqslant (E_2,M_2,F_2)$, the inequality being understood component by component.
\end{lemma}
\begin{proof}
According to Proposition~\ref{prop:4 eq steady}, we have $(E,M,F)\in [0,\bar{E}]\times[0,\bar{M}]\times[0,\bar{F}]$.
  Noting that 
  $$
  M_{s_i}(t)=\int_0^t  u_i(s) e^{\delta_s (s-t)} \,ds, \quad i=1,2,
  $$
  it follows that $M_{s1} \geqslant M_{s2}$. 
Hence, the monotonicity property follows since one has  
$$
\frac{\pa f_E}{\pa F}\geqslant 0, \quad \frac{\pa f_M}{\pa E} \geqslant 0, \quad \frac{\pa f_F}{\pa E} \geqslant 0, \quad \frac{\pa f_F}{\pa M}\geqslant 0, \quad \frac{\pa f_F}{\pa M_s} \leqslant 0,
$$
and therefore the so-called Kamke-M\"uller conditions (see e.g. \cite{HirSmi.MDS}) hold true.
\end{proof}

Let us investigate the existence of an optimal control for Problem~\eqref{eq:opt1}.
\begin{lemma}\label{lemma:4 eq exist}
Let $\varepsilon\in(0,\bar{F})$ and $\overline{U}>U^*$. There exists $\bar{T}(\bar U)>0$ such that for all $T\geqslant \bar{T}(\bar U)$, the set ${\mathcal U}_{T,\bar{U},\varepsilon}^{(\mathcal{S}_2)}$ is nonempty and for such a choice of $T$, Problem~\eqref{eq:opt1} has a solution $u^*$.
    
Moreover, one has $J_{T,\overline{U},\varepsilon}^{(\mathcal{S}_2)}\leqslant \bar{U}\, \bar{T}$ (where $J_{T,\overline{U},\varepsilon}^{(\mathcal{S}_2)}$ is defined by \eqref{optval}) and $J_{T,\overline{U},\varepsilon}^{(\mathcal{S}_2)}$ is non-increasing with respect to $T\geqslant \bar{T}$ and $\overline{U}>U^*$.
\end{lemma}
\begin{proof}
 According to Proposition~\ref{prop:4 eq steady}, if $u=\bar{U}$, then $F(t)\to 0$ as $t$ goes to $+\infty$. Hence, for any $\varepsilon\in (0,\bar{F})$ and since $F$ is Lipschitz-continuous, there exists $\bar{T}>0$ such that $F(\bar{T})\leqslant \varepsilon$ meaning that $T\geqslant \bar{T}$, $u=\bar{U} \mathds{1}_{[T-\bar{T},T]} \in \mathcal{U}^{(\mathcal{S}_2)}_{T,\bar{U},\varepsilon}$. The set $\mathcal{U}^{(\mathcal{S}_2)}_{T,\bar{U},\varepsilon}$ is thus nonempty.

Let us now investigate the existence property. For the sake of clarity, we temporarily denote by  $(E^u,M^u,F^u,M_s^u)$ the solution of System~\eqref{eq:S1} associated to $u$.
Let $(u_n)_{n\in \mathbb{N}}$ be a minimizing sequence. 
According to the Banach-Alaoglu Bourbaki theorem {(see e.g. \cite{rudin1991functional})}, the set $\mathcal{U}_{T,\overline{U},\varepsilon}^{(\mathcal{S}_2)}$ is compact for the weak-$*$ topology of $L^{\infty}(0,T)$ and, up to a subsequence, $(u_n)_{n\in \mathbb{N}}$ converges to $u\in L^{\infty}(0,T,[0,\overline{U}])$. The functional $J$ is obviously continuous for the weak-$*$ topology of $L^{\infty}(0,T)$, so that it only remains to prove that $u\in \mathcal{U}_{T,\overline{U},\varepsilon}^{(\mathcal{S}_2)}$, in other words that $F^u(T)\leqslant \varepsilon$. First, one has $$M_s^{u_n}(t)=\int_0^te^{-\delta_s(t-s)}u_n(s)\, ds$$
for all $t\geqslant 0$ and it thus follows that, $(M_s^{u_n})_{n\in \NN}$ is bounded in $W^{1,+\infty}(0,T)$. Thus, it converges up to a subsequence to $M_s^u$ strongly in $C^0([0,T])$ according to the Ascoli theorem {(see e.g. \cite{rudin1991functional})}. By using the same reasoning as above, the sequence $(E^{u_n},M^{u_n},F^{u_n})$ is non-negative, uniformly bounded by above and therefore, the right-hand side of the first three equations of \eqref{eq:S1} is bounded in $W^{1,+\infty}(0,T)$. Using to the Ascoli theorem, we infer that $(F^{u_n})_{n\in\NN}$ converges to $F^u$ in $C^0([0,T])$. The desired conclusion follows by passing to the limit in the inequality $F^{u_n}(T) \leqslant \varepsilon$.

Finally, the monotonicity of $J_{T,\bar{U},\varepsilon}^{(\mathcal{S}_2)}$ with respect to $\bar{U}$ comes from the monotonicity of the admissible control set $\mathcal{U}^{(\mathcal{S}_2)}_{T,\bar{U},\varepsilon}$ with respect to $\bar U$ and $T$ for the inclusion. More precisely, if $\bar{U}_1\leqslant \bar{U}_2$ then $\mathcal{U}^{(\mathcal{S}_2)}_{T,\bar{U}_1,\varepsilon}\subset \mathcal{U}^{(\mathcal{S}_2)}_{T,\bar{U}_2,\varepsilon}$ according to Lemma~\ref{monotonu}.
Moreover, if $T_1\leqslant T_2$, let $u^*_1$ {be} a solution of Problem~\eqref{eq:opt1}. Then, {$u_2=u^*_1 (\cdot -T_2+T_1)  \mathds{1}_{(T_2-T_1,T_2)} \in \mathcal{U}_{T_2,\bar{U},\varepsilon}$} (since $(E,M,F)$ is stationary on $[0,T_2-T_1)$), and is such that $J(u_2) = J_{T_1,\bar{U},\varepsilon}^{(\mathcal{S}_2)}$. Hence, we get by minimality that $J_{T_2,\bar{U},\varepsilon}^{(\mathcal{S}_2)} \leqslant J_{T_1,\bar{U},\varepsilon}^{(\mathcal{S}_2)}$.

Finally, since $\bar u=\bar{U} \mathds{1}_{[T-\bar{T},T]}$ belongs to $\mathcal{U}^{(\mathcal{S}_2)}_{T,\bar{U},\varepsilon}$ for any $\bar{U}\geqslant U^*$ and $T\geqslant \bar{T}(\bar U)$, one has $J_{T,\bar{U},\varepsilon}^{(\mathcal{S}_2)} \leqslant J(\bar u)= \bar{U}\, \bar{T}$.
\end{proof}

In the following result, one shows  that the state constraint is reached by any optimal control.
\begin{lemma}\label{lemma:satur 4 eq}
Let $\bar U> U^*$ and $T\geqslant \bar T(\bar U)$. Let $u^*$ solving Problem~\eqref{eq:opt1} and $(E,M,F,M_s)$ be the corresponding solution to System \eqref{eq:S1}. 
Then, one has necessarily $F(T)=\varepsilon$ and $F(\cdot)\in(\eps,\bar F]$ on $[0,T)$. 
\end{lemma}
\begin{proof}
Let $u^*$ be a solution to Problem~\eqref{eq:opt1}. Since $\bar{F}$ is a stationary solution and $F(0)=\bar{F}$, one has $F\leqslant \bar{F}$ on $(0,T)$ by using Proposition~\ref{prop:2 eq steady}.
{It remains} to prove that $F$ is bounded below by $\varepsilon$.
 Let us assume by contradiction that the corresponding solution $(E,M,F,M_s)$ to System \eqref{eq:S1} satisfies $F(T)<\varepsilon$.
Let $u_{\eta}:=(1-\eta) u^*$ with $\eta\in(0,1)$ and denote by $(E_{\eta},M_{\eta},F_{\eta},M_{s\eta})$ the corresponding solution to System~\eqref{eq:S1}. Then, by mimicking the arguments of the proof of Lemma~\ref{lemma:4 eq exist}, one easily gets that the mapping $ u\in {\mathcal U}_{T,\bar{U},\varepsilon}^{(\mathcal{S}_2)} \mapsto (E,M,F,M_s)\in C^0([0,T])^4$ is continuous for the weak-$*$ topology of $L^\infty$ and it follows that $F_{\eta}(T)=F(T)+\operatorname{O}(\eta)$ so that $F_{\eta}(T)<\varepsilon$ whenever $\eta$ is small enough. Since  $\int_0^Tu_{\eta}(t)\,dt <\int_0^Tu^*(t)\,dt$, this contradicts the minimality of $u^*$. 

Finally, we claim that $F(\cdot)>\eps$ on $[0,T)$. In the converse case, there exists $T'<T$ such that $F(T')=\eps$. One sees easily that the control $u_{T'}$ defined by $u_{T'}=u^*(\cdot-T+T')\mathds{1}_{[T-T',T]}$ belongs to ${\mathcal U}_{T,\bar{U},\varepsilon}^{(\mathcal{S}_2)}$ and that $J(u_{T'})<J(u^*)$, whence a contradiction.
\end{proof}

Let us now state the necessary optimality conditions for Problem~\eqref{eq:opt1}. To this aim, let us introduce $(P,Q,R,S)$ as the solution to the backward adjoint system 
\begin{equation}\label{eq:dual4}
\left\{\begin{array}{l}
-\ds \frac{d }{d t}\begin{pmatrix}P\\ Q\\R\\S\end{pmatrix} = \left(\begin{array}{cccc}\frac{\partial f_E}{\partial E}(E,F)&\frac{\partial f_M}{\partial E}(E,M)&\frac{\partial f_F}{\partial E}(E,M,F,M_s) & 0 \\
0&\frac{\partial f_M}{\partial M}(E,M)&\frac{\partial f_F}{\partial M}(E,M,F,M_s) & 0 \\
\frac{\partial f_E}{\partial F}(E,F)&0&\frac{\partial f_F}{\partial F}(E,M,F,M_s) & 0 \\
0&0&\frac{\partial f_F}{\partial M_s}(E,M,F,M_s)  & -\delta_s
\end{array}\right)\begin{pmatrix}P\\ Q\\R\\S\end{pmatrix},\\
P(T)=0,Q(T)=0,R(T)=1,~S(T)=0.
\end{array}\right.
\end{equation}
{Let us first determine the differential of $F(T)$ with respect to the control $u$.}
\begin{lemma}\label{lemma:gat 4eq}
Let $u\in \mathcal{U}_{T,\overline{U},\varepsilon}^{(\mathcal{S}_2)}$ and introduce the functional $G$ defined by $G(u)=F(T)$, where $(E,M,F,M_s)$ denotes the unique solution of System \eqref{eq:S1} associated to $u$.
Then, $G$ is differentiable in the sense of Fr\'echet and for every admissible perturbation\footnote{More precisely, we call ``admissible perturbation'' any element of the tangent cone $\mathcal{T}_{u,\mathcal{U}_{T,\overline{U},\varepsilon}^{(\mathcal{S}_2)}}$ to the set $\mathcal{U}_{T,\overline{U},\varepsilon}^{(\mathcal{S}_2)}$ at $u$. Recall that the cone $\mathcal{T}_{u,\mathcal{U}_{T,\overline{U},\varepsilon}^{(\mathcal{S}_2)}}$ is the set of functions $h\in L^\infty(0,T)$ such that, for any sequence of positive real numbers $\varepsilon_n$ decreasing to $0$, there exists a sequence of functions $h_n\in L^\infty(0,T)$ converging to $h$ as $n\rightarrow +\infty$, and $u+\varepsilon_nh_n\in\mathcal{U}_{T,\overline{U},\varepsilon}^{(\mathcal{S}_2)}$ for every $n\in\NN$ (see e.g. \cite{MR1367820}).} $h$, the G\^ateaux-derivative of $G$ at $u$  in the direction $h$ is 
\begin{equation*}
DG(u)\cdot h =\int_0^Th(t)S(t)dt,
\end{equation*}
where $(P,Q,R,S)$ solves System \eqref{eq:dual4}.
\end{lemma}

\begin{proof}
The Fr\'echet-differentiability of $G$ is standard and follows from the differentiability of the mapping $ u\in \mathcal{U}_{T,\overline{U},\varepsilon}^{(\mathcal{S}_2)} \mapsto F$, where $(E,M,F,M_s)$ denotes the unique solution of System \eqref{eq:S1}, itself deriving from a standard application of the implicit function theorem combined with variational arguments. 

Moreover, the Fr\'echet-derivative $(\dot E,\dot M,\dot F,\dot M_s)$ of $(E,M,F,M_s)$ at $u$ in the direction $h$ solves the linearized problem
\begin{equation*}
\left\{\begin{array}{l}
\ds \frac{d }{d t}\begin{pmatrix} \dot E \\ \dot M \\ \dot F \\ \dot M_s\end{pmatrix} = A\begin{pmatrix} \dot E \\ \dot M \\ \dot F \\ \dot M_s\end{pmatrix}+\begin{pmatrix} 0 \\ 0 \\ 0 \\ h\end{pmatrix},\\
\dot E(0)=0,~\dot M(0)=0,~\dot F(0)=0,~\dot M_s(0)=0,
\end{array}\right.
\end{equation*}
where the Jacobian matrix $A$ reads 
$$
A=\begin{pmatrix} 
\scriptstyle\frac{\partial f_E}{\partial E}(E,F)&\scriptstyle0&\scriptstyle\frac{\partial f_E}{\partial F}(E,F)&\scriptstyle0\\
\scriptstyle\frac{\partial f_M}{\partial E}(E,M)&\scriptstyle\frac{\partial f_M}{\partial M}(E,M)&\scriptstyle0&\scriptstyle0\\
\scriptstyle\frac{\partial f_F}{\partial E}(E,M,F,M_s)&\scriptstyle\frac{\partial f_F}{\partial M}(E,M,F,M_s)&\scriptstyle\frac{\partial f_F}{\partial F}(E,M,F,M_s)&\scriptstyle\frac{\partial f_F}{\partial M_s}(E,M,F,M_s)\\
\scriptstyle0&\scriptstyle0&\scriptstyle0 &\scriptstyle -\delta_s\end{pmatrix}.
$$
By integration by parts, it follows that
\begin{eqnarray*}
\int_0^T \left\langle \begin{pmatrix} 0\\ 0\\ 0\\ h\end{pmatrix},\begin{pmatrix} P\\Q\\R\\S\end{pmatrix}\right\rangle&=&
\int_0^T \left\langle \frac{d}{dt}\begin{pmatrix} \dot E \\ \dot M \\ \dot F \\ \dot M_s\end{pmatrix},\begin{pmatrix} P\\Q\\R\\S\end{pmatrix}\right\rangle 
- \int_0^T \left\langle A\begin{pmatrix} \dot E \\ \dot M \\ \dot F \\ \dot M_s\end{pmatrix},\begin{pmatrix}P\\Q\\R\\S\end{pmatrix}\right\rangle\\
&=& -\int_0^T \left\langle \begin{pmatrix} \dot E \\ \dot M \\ \dot F \\ \dot M_s\end{pmatrix},\frac{d}{dt}\begin{pmatrix} P\\Q\\R\\S\end{pmatrix}\right\rangle + \left[\left\langle \begin{pmatrix} \dot E \\ \dot M \\ \dot F \\ \dot M_s\end{pmatrix},\begin{pmatrix}P\\Q\\R\\S\end{pmatrix}\right\rangle\right]_0^T\\
&&- \int_0^T \left\langle \begin{pmatrix} \dot E \\ \dot M \\ \dot F \\ \dot M_s\end{pmatrix},A^T\begin{pmatrix}P\\Q\\R\\S\end{pmatrix}\right\rangle= F(T),
\end{eqnarray*}
which leads to the desired result.
\end{proof}
{\begin{remark}
This result, as well as the next one, can also be obtained by using the so-called Pontryagin Maximum Principle (see for instance \cite{Lee-Markus}). In particular, the final condition for the adjoint state is called {\it transversality condition}.
\end{remark}}
The following Lemma leads to a characterization of any optimal control. We chose here to provide some quick explanations on the derivation of optimality conditions. It would have also been possible to use a shorter argument through the so-called Pontryagin Maximum Principle (PMP) and we would have obtained the same result.
\begin{lemma}\label{lemma:cond_opt 4eq}
Let $\bar U> U^*$ and $T\geqslant \bar T(\bar U)$.
Let  $u^*$ denote a solution to Problem~\eqref{eq:opt1}.
There exists $\lambda> 0$ such that
\begin{equation*}
\left\{\begin{array}{l}
\text{a.e. on }\{u^*=0\}\text{, one has } 1+\lambda S(t)\geqslant 0,\\
\text{a.e. on }\{0<u^*<\overline{U}\}\text{, one has } 1+\lambda S(t)=0,\\
\text{a.e. on }\{u^*=\overline{U}\}\text{, one has } 1+\lambda S(t)\leqslant 0.
\end{array}\right.
\end{equation*}
\end{lemma}
\begin{proof}

Let us introduce the Lagrangian function $\mathcal{L}$ associated to problem \eqref{eq:opt1}, defined by 
{\begin{equation*}
\mathcal{L}:(u,\lambda)\in\mathcal{U}_{T,\overline{U}}\times \mathbb{R}\mapsto J(u)-\lambda(F(T)-\varepsilon),
\end{equation*}}
where $\mathcal{U}_{T,\overline{U}}:=\{u\in L^{\infty}(0,T):0\leqslant u(\cdot)\leqslant \overline{U}\}$.

By standard  arguments, we get  the  existence  of  a  Lagrange  multiplier  $\lambda\geqslant 0$ such that $(u^*,\lambda)$ satisfies $D_u\mathcal{L}(u^*,\lambda)\cdot h\geq 0$ for every $h$ belonging to the tangent cone of the set $\mathcal{U}_{T,\overline{U}}^{(\mathcal{S}_2)}$ at $u^*$.
Moreover, according to Lemma \ref{lemma:satur 4 eq}, we have necessarily $F(T)=\varepsilon$. 

Let $t^*$ be a Lebesgue density-one point of $\{u^*=0\}$.  Let $(H_n)_{n\in\mathbb{N}}$ be a sequence of measurable  subsets containing all $t^*$ and such that $H_n$ is included  in  $\{u^*=0\}$.  
 Let  us  consider $h=\mathds{1}_{H_n}$ and  notice  that, by construction, $u^*+\eta h$ belongs  to $\mathcal{U}_{T,\overline{U}}^{(\mathcal{S}_2)}$ whenever $\eta$ is  small  enough.
One has
 $$\mathcal{L}(u^*+\eta h,\lambda)\geqslant \mathcal{L}(u^*,\lambda),$$
 whenever $\eta$ is  small  enough. Let us divide this inequality by $\eta$, 
 and let $\eta$ go to 0. By using Lemma~\ref{lemma:gat 4eq}, we obtain
$$\int_0^Th(t)dt+ \lambda \int_0^Th(t)S(t)dt\geqslant 0,$$
which rewrites $|H_n|+ \lambda \int_{H_n}S(t)dt\geqslant 0.$ Dividing  this  inequality  by $|H_n|$ and  letting
$H_n$ shrink  to $\{t^*\}$ as $n\rightarrow\infty$ shows that $1+\lambda S(t)\geqslant 0$ on $\{u^*=0\}$. This proves the  first  point  of  Lemma~\ref{lemma:cond_opt 4eq},  according  to  the  Lebesgue Density Theorem {(see e.g. \cite{mattila1999geometry})}.  The proof of the third point is similar and consists in considering perturbations of the form $u^*-\eta h$, where $h$ denotes a positive admissible perturbation of $u^*$ supported in $\{u^* = \overline{U}\}$.  
Finally, the proof of the second point follows the same lines by considering bilateral perturbations of the form $u^*\pm\eta h$, where $h$ denotes an admissible perturbation of $u^*$  supported in $\{0< u^*< \overline{U}\}$.

Let us now prove that $\lambda>0$. We argue by contradiction, assuming that $\lambda=0$. Then, the switching function $1+\lambda S$ is necessarily constant, equal to 1, and we have therefore $u^*=0$ in $[0,T]$, which leads to a contradiction since the optimal trajectory has to satisfy $F(T)=\varepsilon$.
\end{proof}


Let us prove the remaining facts stated in Theorem~\ref{theo:opt 1}.
\begin{proof}[Proof of Theorem~\ref{theo:opt 1}]
Let $\varepsilon\in(0,\bar F)$. According to Lemma~\ref{lemma:4 eq exist}, there exists $U^*$ such that for each $\bar{U}>U^*$, there exists $\bar T$ such for all $T>\bar T$, Problem~\eqref{eq:opt1} has a solution $u^*$. Let {$(E,M,F,M_s)$} be the optimal trajectory.
According to Lemma~\ref{lemma:satur 4 eq}, the constraint ``$F(T)=\varepsilon$'' is reached.
Lemma~\ref{lemma:cond_opt 4eq} implies that on $\{S>-1/\lambda\}$, one has necessarily $u=0$. Since $S$ is continuous and $S(T)=0$, it follows that there exists $t_1\in(0,T)$ such that $u^*=0$ on $(t_1,T)$.
\end{proof}

\subsection{Proof of Theorem \ref{theo:opt 2}}
Let us first point out that System~\eqref{eq:primal2} is monotone with respect to the control $u$~:
\begin{lemma}\label{monotonu2}
  Let $u_1, u_2\in L^{\infty}(0,T)$ such that  $u_1\geqslant u_2\geqslant 0$ (resp. $u_1>u_2\geqslant 0$). Let us assume that \eqref{init} holds.
  Then, the corresponding solutions $F_1$, $F_2$ to System~\eqref{eq:primal2} satisfy $F_1\leqslant F_2$ on $(0,T)$ (resp. $F_1< F_2$ on $(0,T)$).
\end{lemma}
\begin{proof}
This is an immediate consequence of the fact that $\frac{\pa f}{\pa M_s} \leqslant 0$ when $F\in (0,\bar{F}]$.
\end{proof}

Let us investigate the existence of an optimal control for Problem~\eqref{eq:opt2}.
\begin{lemma}\label{lemma:2 eq exist}
    Let $\varepsilon\in (0,\bar{F})$. For every $\overline{U}>U^*$, there exists $\bar{T}(\bar U)>0$ such that for all $T\geqslant \bar T(\bar U)$, the set ${\mathcal U}_{T,\bar{U},\varepsilon}^{(\mathcal{S}_1)}$ is {nonempty} and Problem~\eqref{eq:opt2} has a solution $u^*$.

Moreover, one has $J_{T,\bar{U},\varepsilon}^{(\mathcal{S}_1)} \leqslant \bar{U}\,\bar{T}$ (where $J_{T,\overline{U},\varepsilon}^{(\mathcal{S}_1)}$ is defined by \eqref{optval}) and $J_{T,\bar{U},\varepsilon}^{(\mathcal{S}_1)}$ is non-increasing with respect to $T\geqslant \bar{T}$ and to $\bar{U}>U^*$.
\end{lemma}
\begin{proof}
  We first prove that ${\mathcal U}_{T,\bar{U},\varepsilon}^{(\mathcal{S}_1)}$ is nonempty when $\bar{U}>U^*$ and $T$ is large enough. If $u(\cdot)=\bar{U}>U^*$, we have seen in Proposition~\ref{prop:2 eq steady} that $F(t)\to 0$ as $t\to +\infty$. Thus, for any $\varepsilon\in(0,\bar{F})$, there exists $\bar{T}>0$, such that $F(\bar{T})=\varepsilon$.
  Hence, for $T\geqslant \bar{T}$, $u=\bar{U} \mathds{1}_{[T-\bar{T},T]}$ belongs to $\mathcal{U}^{(\mathcal{S}_1)}_{T,\bar{U},\varepsilon}$.

Proceeding as in the proof of Lemma~\ref{lemma:4 eq exist}, we obtain existence of a solution to Problem~\eqref{eq:opt2} by considering a minimizing sequence and showing that it is in fact compact.
  Finally, the monotonicity and the bound on $J_{T,\bar{U},\varepsilon}^{(\mathcal{S}_1)}$ are obtained exactly as in the end of the proof of Lemma~\ref{lemma:4 eq exist}.
\end{proof}

By mimicking the proof of Lemma~\ref{lemma:satur 4 eq} for Problem~\eqref{eq:opt1}, we can prove that the constraint $F(T)\leqslant \varepsilon$ is saturated.

\begin{lemma}\label{lemma:satur}
Let $\bar U> U^*$, $T\geqslant \bar T(\bar U)$ and $u^*$ be a solution to Problem~\eqref{eq:opt2}. Let $(F,M_s)$ be the associated optimal trajectory, solution to System~\eqref{eq:primal2}. 
Then, one has $F(T)=\varepsilon$  and $F(\cdot)\in(\eps,\bar F]$ on $[0,T)$. 
\end{lemma}
We can also prove that $F$ is non-increasing on $(0,T)$. 

%


\begin{lemma}\label{lem:f1neg}
Let $\bar U> U^*$, $T\geqslant \bar T(\bar U)$ and $u^*$ be a solution to Problem~\eqref{eq:opt2}. Let $(F,M_s)$ be the associated optimal trajectory, solution to System~\eqref{eq:primal2}.
Then, for every $t\in (0,T)$, one has $F'(t) \leqslant 0$.
\end{lemma}
\begin{proof}
Let us argue by contradiction, assuming that the conclusion is not true. Then, since $F$ is $C^1$, there exist $0<\theta_1<\theta_2<T$ such that $F'>0$ on $(\theta_1,\theta_2)$.

According to Lemma~\ref{monotonu2}, $F$ is non-increasing in a neighborhood of 0. Moreover, according to Lemma~\ref{lemma:satur}, $F$ decreases to $\eps$ in a neighborhood of $T$. Consequently, it is not restrictive to also assume that $F'(\theta_1)=F'(\theta_2)=0$, and $F'\leqslant 0$ on $(0,\theta_1)$ 
 
 Notice from the expression of $f$ given in \eqref{eq:f1} that $F'=f(F,M_s)\leqslant 0$ if, and only if,
  \begin{equation*}
\delta _M \gamma_s M_s\geqslant\phi(F)=(1-\nu)\beta_E \nu_E\frac{ \nu \beta_E \nu_E
- \delta_F \big(\frac{\beta_E F}{K} + \nu_E+\delta_E \big) }
{ \delta_F \big(\frac{\beta_E F}{K} + \nu_E+\delta_E \big)^2 }F.
\end{equation*}
 
  Let us show that there exist $0<\tau_1<\tau_2<T$ such that $F(\tau_1) < F(\tau_2)$ and $M_s(\tau_1)\geqslant M_s(\tau_2)$. Indeed, there are two possibilities. Either there exists $\tau_2\in (\theta_1,\theta_2)$ such that $M_s(\theta_1) = M_s(\tau_2)$, and then we take $\tau_1=\theta_1$, or for any $t\in(\theta_1,\theta_2)$, $M_s(\theta_1)<M_s(t)$. 
  In this latter case, we take $\tau_2\in(\theta_1,\theta_2)$ such that $M_s(\theta_1)<M_s(\tau_2)<\phi(F(\tau_2))$ (which is always possible since $F' > 0$ on $(\theta_1,\theta_2)$). 
  Then, since $F(0)=\bar{F}>F(\tau_2)>F(\theta_1)$ and $F$ is continuous, there exists $\tilde{\tau}\in (0,\theta_1)$ such that $F(\tilde{\tau}) = F(\tau_2)$.
  Moreover, since $M_s\geqslant \phi(F)$ on $(0,\theta_1)$, we have $M_s(\tilde{\tau})\geqslant \phi(F(\tilde{\tau}))=\phi(F(\tau_2)) > M_s(\tau_2)>M_s(\theta_1)$.
  By continuity of $M_s$, there exists $\tau_1 \in (\tilde{\tau},\theta_1)$ such that $M_s(\tau_1)=M_s(\tau_2)$.
  Since we have $F'\leqslant 0$ on $(0,\theta_1)$, we deduce that $F(\tau_1) \leqslant F(\tilde{\tau}) = F(\tau_2)$ and $F(\tau_2)\neq F(\tau_1)$ since $\phi(F(\tau_1))\leqslant M_s(\tau_1)=M_s(\tau_2)<\phi(F(\tau_2))$.
  
  Then, we take 
\begin{equation*}
u(t)=\left\{\begin{array}{ll}
0,&\mbox{for }t\in(0,\tau_2-\tau_1),\\
u^*(t-\tau_2+\tau_1),&\mbox{for }t\in(\tau_2-\tau_1,\tau_2),\\
u^*(t),&\mbox{for }t\in(\tau_2,T).
\end{array}\right.
\end{equation*}  
Using Lemma \ref{lemma:Q,R}, we have $J(u)\leqslant J(u^*)$.
  Moreover, if we denote by $(F^u,M_s^u)$ the solution with this control $u$, we have $F^u(\tau_2) = F(\tau_1) < F(\tau_2)$ and $M_s^u(\tau_2)=M_s(\tau_1)\geqslant M_s(\tau_2)$. Then, on $(\tau_2,T)$, we have $M_s^u\geqslant M_s$ and since $M_s\mapsto f(F,M_s)$ is non-increasing $(F^u)' = f(F^u,M_s^u) \leqslant f(F^u,M_s)$. By comparison with the solution to $F'=f(F,M_s)$ on $(\tau_2,T)$ and since $F^u(\tau_2)<F(\tau_2)$, we deduce that $F^u(T) < F(T) = \varepsilon$ which contradicts Lemma~\ref{lemma:satur}.
\end{proof}

Let us introduce $(Q,R)$ as the solution of the forward adjoint system 
\begin{equation}\label{eq:dual3}
\left\{\begin{array}{l}
-\ds \frac{d }{d t}\begin{pmatrix} Q\\R\end{pmatrix} = \left(\begin{array}{ccc}\frac{\partial f}{\partial F}(F,M_s) & 0 \\
\frac{\partial f}{\partial M_s}(F,M_s) & -\delta_s
\end{array}\right)\begin{pmatrix} Q \\ R \end{pmatrix},\\
Q(T)=1,~R(T)=0.
\end{array}\right.
\end{equation}

Similarly to Problem~\eqref{eq:opt1}, any optimal control can be characterized by using the first order necessary optimality conditions, in terms of a switching function of the form $t\mapsto 1+\lambda R(t)$ with $\lambda\geq 0$.
\begin{lemma}\label{lemma:cond_opt 2eq}
Let $\bar U> U^*$ and $T\geqslant \bar T(\bar U)$.
Consider  $u^*$ a solution to Problem~\eqref{eq:opt2}.
Then there exists $\lambda> 0$ such that
\begin{equation*}
\left\{\begin{array}{l}
\text{a.e. on }\{u^*=0\}\text{, one has } 1+\lambda R(t)\geqslant 0,\\
\text{a.e. on }\{0<u^*<\overline{U}\}\text{, one has } 1+\lambda R(t)=0,\\
\text{a.e. on }\{u^*=\overline{U}\}\text{, one has } 1+\lambda R(t)\leqslant 0.
\end{array}\right.
\end{equation*}
\end{lemma}
The proof is similar to the one of Lemma~\ref{lemma:cond_opt 4eq}.

\begin{lemma}\label{lemma:0 3 eq}
Let $\bar U> U^*$ and $T\geqslant \bar T(\bar U)$. There exists $t_1\in(0,T)$ such that
$u^*=0\mbox{ on }(t_1,T)$.
\end{lemma}
\begin{proof}
According to Lemma~\ref{lemma:cond_opt 2eq}, on the set $\{R>-1/\lambda\}$, one has necessarily $u=0$. We conclude using the fact that $R$ is continuous and $R(T)=0$.
\end{proof}

\begin{lemma}\label{lemma:Q,R}
Let $\bar U> U^*$ and $T\geqslant \bar T(\bar U)$.
There exist positive constants $C_0$, $C_1$, $C_2>0$ that do not depend on $\bar U$ and $T$, such that 
$$
|M_s|<C_0,\quad  0 < C_1 \leqslant Q \leqslant C_2.
$$
\end{lemma}
\begin{proof}
First, one has for all $t\in [0,T]$,
$$
M_s(t)=\int_0^tu^*(s)e^{\delta_s(s-t)}ds\leqslant \int_0^Tu^*(t) \,dt
 \leqslant \bar T\bar U.
$$
From \eqref{eq:dual3}, $Q$ solves
$$
-\frac{dQ }{d t} = \frac{\pa f}{\pa F}(F,M_s) Q  ,\qquad
Q(T) = 1.
$$
Note that, on $[0,T]$, all quantities are bounded: one has $\eps\leqslant F(\cdot)\leqslant F(0)=\bar F$ according to Lemma~\ref{lem:f1neg}. Hence, we infer the existence of $\mu>0$ such that 
$-\mu \leqslant \frac{\partial f}{\partial F}(F,M_s) \leqslant \mu$. Thus, integrating the inequality $-\mu Q\leqslant -\frac{dQ }{d t}  \leqslant \mu Q$ with $Q(T)=1$, the conclusion follows easily by using a Gronwall type argument {(see e.g. \cite{gronwall1919note})}.
\end{proof}

\begin{lemma}\label{lemma:0 3 eq begin}
Let $\bar U> U^*$ and $T\geqslant \bar T(\bar U)$. If $0<u^*<\bar U$ on a nonempty interval $(s_0,s_1)$ then
$u^*$ satisfies \eqref{3eq:edo u} on $(s_0,s_1)$.
\end{lemma}
\begin{proof}
According to Lemma~\ref{lemma:cond_opt 2eq}, one has $R'=0$ on $(s_0,s_1)$.
Differentiating the equation satisfied by $R$ in \eqref{eq:dual3} yields
\begin{equation*}
\left(\frac{\partial^2 f}{\partial M_s\partial F} F'+\frac{\partial^2 f}{\partial M_s^2}M_s'\right)Q+\frac{\partial f}{\partial M_s}Q'=0.
\end{equation*}
We deduce that
\begin{equation*}
\left(\frac{\partial^2 f}{\partial M_s\partial F}f+\frac{\partial^2 f}{\partial M_s^2}(u-\delta_sM_s)\right)Q=\frac{\partial f}{\partial M_s}\frac{\partial f}{\partial F} Q.
\end{equation*}
According to Lemma~\ref{lemma:Q,R}, one has $Q>0$ on $[0,T]$. Hence $u$ is given by 
\begin{equation*}
u = \left(\frac{\partial^2 f}{\partial M_s^2}\right)^{-1} \left(\frac{\partial f}{\partial M_s}\frac{\partial f}{\partial F}+ \frac{\partial^2 f}{\partial M_s^2} \delta_s M_s -\frac{\partial^2 f}{\partial M_s\partial F}f  \right),
\end{equation*}
where $f$ is given by \eqref{eq:f1}. We thus recover \eqref{3eq:edo u}.
\end{proof}

\begin{lemma}\label{lemma:bar U 2 eq}
Let $\bar U> U^*$ and $T\geqslant \bar T(\bar U)$.
\begin{enumerate}
\item If $2\delta_s>\delta_F$, then, on each open interval of $\{R>-1/\lambda\}$, any local extremum for $R$ is a minimum.
\item If $\bar U$ is large enough, then, on each open interval of  $\{R<-1/\lambda\}$, any local extremum for $R$ is a maximum.
\end{enumerate}
\end{lemma}
\begin{proof}
\begin{itemize}
\item[(i)]
By differentiating the equation on $R$ in \eqref{eq:dual3}, we get
\begin{equation}\label{eq:R''}
-R''=\left(\frac{\partial^2 f}{\partial M_s\partial F}f + \frac{\partial^2 f}{\partial M_s^2}(u^*-\delta_s M_s) - \frac{\partial f}{\partial M_s}\frac{\partial f}{\partial F}\right) Q-\delta_s R'\quad \text{on $(0,T)$.}
\end{equation}
Let $I$ be an open interval of $\{R>-1/\lambda\}$ (whenever it exists). 
Let $\tau\in I$ be a local extremum for $R$, we have $R'(\tau)=0$.
Then, thanks to Lemma \ref{lemma:cond_opt 2eq}, we have $u^*=0$ on $I$ and hence
$$
-R''(\tau)=\left(\frac{\partial^2 f}{\partial M_s\partial F}f - \frac{\partial^2 f}{\partial M_s^2}\delta_s M_s - \frac{\partial f}{\partial M_s}\frac{\partial f}{\partial F}\right)(\tau) Q(\tau).
$$
We have seen in Lemma \ref{lemma:Q,R} that $Q>0$.
Then, the sign of $-R''(\tau)$ is the sign of $\mathfrak{U}$ given by
$$
\mathfrak{U}=\frac{\partial^2 f}{\partial M_s\partial F}f - \frac{\partial^2 f}{\partial M_s^2}\delta_s M_s - \frac{\partial f}{\partial M_s}\frac{\partial f}{\partial F}.
$$
Let us compute $\mathfrak{U}$. Notice that $f$ is of the form
\begin{equation}
  \label{eq:Lambda}
  f(F,M_s) = \mu F^2 \Lambda - \delta_F F\qquad \text{where}\qquad 
  \Lambda := \frac{1}{F^2 + a F + M_s(\alpha F^2 + \beta F + \gamma)}
\end{equation}
with some positive constants $\mu, a, \alpha, \beta, \gamma$.
Observe that
$$
\frac{\pa \Lambda}{\pa M_s} = -\Lambda^2(\alpha F^2 + \beta F + \gamma) , \qquad 
\frac{\pa \Lambda}{\pa F} = -  \Lambda^2(2F+a+M_s(2\alpha F + \beta)),
$$
and therefore one computes
\begin{eqnarray*}
\frac{\pa f}{\pa M_s} &=& -\mu F^2 \Lambda^2 (\alpha F^2 + \beta F + \gamma),\\
\frac{\pa f}{\pa F} &=& 2 \mu F \Lambda -\mu F^2 \Lambda^2 (2F+ a + M_s (2\alpha F + \beta))  - \delta_F,\\
\frac{\pa^2 f}{\pa M_s^2} &=& 2 \mu F^2\Lambda^3 (\alpha F^2 + \beta F + \gamma)^2\\
  \frac{\pa^2 f}{\pa M_s\pa F} &=& - \mu F \Lambda^2(2(\alpha F^2 + \beta F + \gamma) + F(2\alpha F + \beta)), \\
 & & + 2 \mu F^2 \Lambda^3 (2F+a+M_s(2\alpha F+\beta))(\alpha F^2 + \beta F + \gamma).
\end{eqnarray*}
After straightforward but tedious computations, we find
\begin{align*}
 & \frac{\partial^2 f}{\partial M_s\partial F}f\ -\  \frac{\partial f}{\partial M_s}\frac{\partial f}{\partial F}\\
 &  =  -\mu F^2\Lambda^3   (2\mu F(\alpha F^2 + \beta F + \gamma) + \mu F^2 (2\alpha F + \beta)) \\
 & ~~~ +2\mu^2 F^4 \Lambda^4 (2F+a+M_s(2\alpha F + \beta)) (\alpha F^2 + \beta F + \gamma)  \\
&~~~   + \delta_F F \Lambda^2 \big(2\mu F(\alpha F^2 + \beta F +\gamma) + \mu F^2 (2\alpha F + \beta)\big) \\
&~~~  - 2\mu \delta_F F^3 \Lambda^3 (2F+a+M_s(2\alpha F + \beta))(\alpha F^2 + \beta F + \gamma)  \\
 &~~~  + \mu F^2 \Lambda^2 (\alpha F^2 + \beta F + \gamma) (2\mu F \Lambda - \delta_F - \mu(2F+a+M_s(2\alpha F+ \beta)) \Lambda^2 F^2) \\
&= \ \mu^2 F^4 \Lambda^4 \Big( (2F+a + M_s(2\alpha F+\beta))(\alpha F^2 + \beta F + \gamma) - \frac{2\alpha F + \beta}{\Lambda} \Big) \\
&~~~ + \mu F^2 \Lambda^3 \delta_F \times  \Big(\frac{3\alpha F^2 + 2 \beta F +\gamma}{\Lambda} 
\\&~~~- 2F(2F+a+M_s(2\alpha F + \beta))(\alpha F^2 +\beta F + \gamma)\Big).
\end{align*}
Using the expression of $1/\Lambda$ from \eqref{eq:Lambda}, we get
\begin{align*}
  \frac{\partial^2 f}{\partial M_s\partial F}f - \frac{\partial f}{\partial M_s}\frac{\partial f}{\partial F}
= & \mu^2 F^4 \Lambda^4 ( (\beta - \alpha a) F^2 + 2\gamma  F + a \gamma)  \\
  & + \mu F^2 \Lambda^3 \delta_F (-\alpha F^4+(a\alpha-2\beta)F^3-3\gamma F^2-a\gamma F)) \\
  & + \mu F^2 \Lambda^3 \delta_F M_s(\alpha F^2+\beta F + \gamma) (\gamma-\alpha F^2).
\end{align*}
Finally, we obtain
\begin{eqnarray*}
\mathfrak{U}&=&  \mu^2 F^4 \Lambda^4 ( (\beta - \alpha a) F^2 + 2\gamma F  + a \gamma)\\&&+ \mu F^3 \Lambda^3 \delta_F (-\alpha F^3+(a\alpha-2\beta)F^2-3\gamma F-a\gamma )) \\
 & & + \mu F^2 \Lambda^3 \delta_F M_s(\alpha F^2+\beta F + \gamma) (\gamma-\alpha F^2) \\&&-\delta_sM_s 2 \mu F^2  \Lambda^3(\alpha F^2 + \beta F + \gamma)^2.
\end{eqnarray*}
Let us use that $2\delta_s>\delta_F$. Noting that $f\leqslant 0$ as a consequence of Lemma~\ref{lem:f1neg}, which rewrites $\mu F \Lambda \leqslant \delta_F$, we obtain 
\begin{eqnarray}\label{eq:est R''}
\mathfrak{U} &\leqslant &- \mu^2 F^4 \Lambda^4   \alpha a F^2+ \mu F^3 \Lambda^3 \delta_F (-\alpha F^3+(a\alpha-\beta)F^2-\gamma F )) \\
 &&  - \mu F^2 \Lambda^3 \delta_F M_s(\alpha F^2+\beta F + \gamma) (2\alpha F^2+\beta F).\nonumber
\end{eqnarray}
We observe that this quantity is negative whenever $\beta>a\alpha$, which is the case since
\begin{equation*}
a=\frac{K(\nu_E+\delta_E)}{\beta_E},\quad 
\alpha=\frac{\delta_M\gamma_s}{K(1-\nu)\nu_E}\quad \mbox{ and }\quad
\beta=\frac{2\delta_M\gamma_s(\nu_E+\tau_E)}{(1-\nu)\beta_E\nu_E}.
\end{equation*}
\item[(ii)] Let $I$ {be} an open interval of $\{R<-1/\lambda\}$. 
Let $\tau\in I$ be a local extremum of $R$. Therefore, we have $R'(\tau)=0$. According to Lemma \ref{lemma:cond_opt 2eq}, one has $u^*=\bar U$ on $I$. 
Hence, from \eqref{eq:R''}, one gets
\begin{equation}\label{eq:R''2}
-R''=\left(\frac{\partial^2 f}{\partial M_s\partial F}f + \frac{\partial^2 f}{\partial M_s^2}(\bar{U}-\delta_s M_s) - \frac{\partial f}{\partial M_s}\frac{\partial f}{\partial F}\right) Q-\delta_s R'.
\end{equation}
Recall that $\eps\leqslant F(\cdot)\leqslant F(0)=\bar F$ according to Lemma~\ref{lem:f1neg}. By using also Lemma~\ref{lemma:Q,R}, we get the existence of $C>0$ independent of $T$ and $\bar U$ such that
\begin{equation}\label{eq:est fMM}\begin{array}{rcl}
\frac{\partial^2 f}{\partial M_s^2}(F,M_s)& =&  \frac{2\nu(1-\nu)\beta_E^2 \nu_E^2 F^2\delta_M^2 \gamma_s^2(\frac{\beta_E F}{K} + \nu_E+\delta_E)}{\big((1-\nu)\nu_E \beta_E F+\delta_M \gamma_s M_s(\frac{\beta_E F}{K} + \nu_E+\delta_E)\big)^3}\\
&\geqslant & \frac{2\nu(1-\nu)\beta_E^2 \nu_E^2 F^2\delta_M^2 \gamma_s^2(\nu_E+\delta_E)}{\big((1-\nu)\nu_E \beta_E \bar F+\delta_M \gamma_s C_0(\frac{\beta_E \bar F}{K} + \nu_E+\delta_E)\big)^3}>C.\end{array}
\end{equation} 
A similar reasoning shows that all the other terms in \eqref{eq:R''2} are uniformly bounded with respect to $T$ and $\bar U$. Thus, we can find $\bar{U}$ (independent on $T$) large enough such that the right-hand side is positive.
Then, $R''(\tau)< 0$, which implies $R$ admits a local maximum at $\tau$.
\end{itemize}
\end{proof}
\begin{lemma}\label{lem:u=0}
Let us make the same assumption as in Proposition~\ref{prop:2 eq steady}. Consider the dynamical system \eqref{eq:primal2} with $u=0$, i.e.
  $$
  \frac{dF}{dt} = f(F,M_s), \qquad  \frac{dM_s}{dt} = - \delta_s M_s,
  $$
  where $f$ is given in \eqref{eq:f1}.
  If $F'(\tau)\geqslant 0$ then, for all $t> \tau$, we have $F'(t)\geqslant 0$. 
\end{lemma}
\begin{proof}
  This is a consequence of the fact that the set 
  $$
  \mathcal{E}:=\{(F,M_s)\in \RR_+\times (0,+\infty), \text{ such that } f(F,M_s) \geqslant 0\}
  $$ 
  is stable by the aforementioned dynamical system. Indeed, on {$\RR_+^2$}, $f(F,M_s) \geqslant 0$ if, and only if, $0 \leqslant M_s \leqslant \phi(F)$, where the function $\phi$ is obtained by solving the implicit equation $f(F,\phi(F))=0$. 
 If at some time $\tau>0$, a trajectory crosses the part of the boundary of $\mathcal E$ defined by the implicit equation, the vector field defining the right-hand side of the differential system reads
 $$
 [f(F(\tau),M_s(\tau)),-\delta_sM_s(\tau)]^\top=[0,-\delta_sM_s(\tau)]^\top.
 $$  
 This field is vertically directed, and it points inward for $\mathcal E$. The conclusion is similar on the other parts of the boundary of $\mathcal{E}$, whence the stability of this zone.  
\end{proof}

\begin{proof}[Proof of Theorem \ref{theo:opt 2}]
We separately deal with the characterization and uniqueness properties minimizers.

\medskip

\noindent \textbf{Step 1: characterization of optimal controls.} 
Notice first that the sets $\{R>-1/\lambda\}$ and $\{R<-1/\lambda\}$ are open, as inverse images of open sets by continuous functions (and thus contain an interval).
By combining Lemma~\ref{lemma:bar U 2 eq}, Lemma~\ref{lemma:0 3 eq}, and the fact that $R$ is continuous with $R(T)=0$, we get  
the existence of $(t_0,t_1)\in[0,T]^2$ such that $0\leq t_0\leq t_1 < T$ and 
\begin{itemize}
\item $R>-1/\lambda$ on $(0,t_0)$ or $R<-1/\lambda$ on $(0,t_0)$;
\item $R=-1/\lambda$ on $(t_0,t_1)$;
\item $R>-1/\lambda$ on $(t_1,T)$. 
\end{itemize}
Combined with Lemmas~\ref{lemma:cond_opt 2eq} and \ref{lemma:0 3 eq begin}, we deduce the form of the optimal control in 
Theorem \ref{theo:opt 2}.


Let us now prove that for $T$ large enough, we have $t_0>0$ and $u^*=0$ on $(0,t_0)$.
Assume by contradiction that $u^*=\bar U$ on $(0,t_0)$ or $t_0=0$.
Recall that, according to Lemma~\ref{lemma:2 eq exist}, we have
  $$
  J_{T,\bar{U},\varepsilon}^{(\mathcal{S}_1)} = \int_0^T u^*(t)\,dt \leqslant \bar{U}\,\bar{T}.
  $$
 However, combining the expression of $u^*$ given in \eqref{3eq:edo u} with the estimates \eqref{eq:est R''} and  \eqref{eq:est fMM} show that
  $$
  u^*(t) \geqslant \tilde C>0,\quad \text{ on $(t_0,t_1)$},
  $$
  where $\tilde C$ does not depend on {neither $T$ nor} $\bar U$. Without loss of generality, we can assume that $\tilde C<U^*$. Hence, one has 
  $$
  \int_0^{t_1}u^*(t)\,dt \geqslant t_1\tilde C
  $$
(since $u^*=\bar U>U^*>\tilde C$ on $(0,t_0)$).
It follows that $t_1\leqslant \bar{U}\bar{T}/\tilde{C}$ is uniformly bounded with respect to $T$. On $(t_1,T)$, one has $u^*=0$. 
 Let $F_{\bar U}$ denote the trajectory associated to the control choice $u_{\bar U}=\bar U\mathds{1}_{[0,\bar U\bar T/\tilde C]}$. According to the monotonicity property stated in Lemma~\ref{monotonu2}, one has $F\geqslant F_{\bar U}$ in $\RR_+$ since $u\leqslant u_{\bar U}$ . Furthermore, according to Proposition~\ref{prop:2 eq steady}, $F_{\bar U}(T)$ converges to the steady state $\bar{F}$ as $T\to +\infty$. Let $T_{\bar U}>0$ be given such that $F_{\bar U}(t)\geqslant (\eps+\bar F)/2>\eps$ in $(T_{\bar U},+\infty)$. If $T>T_{\bar U}$, one thus has $F(T)>\eps$ which contradicts the fact that $F(T)=\eps$ (see Lemma \ref{lemma:satur}). 
 
 Hence, there exists $T^*>0$ large enough such that, for $T>T^*$, we have $t_0>0$, and $u^*=0$ on $(0,t_0)$. Necessarily, $t_1>t_0$, otherwise $u^*=0$ a.e. on $(0,T)$ which is not possible, since in this case $F(t)=\bar{F}>\varepsilon$ on $(0,T)$.

  Let us notice that if $T>T^*$, we have
   \begin{equation}\label{eq:F(T)}
   F'(T)\geqslant 0.
  \end{equation} 
  Indeed, since $T>T^*$, we have seen that the optimal control has the form $u^*=u^* \mathds{1}_{(t_0,t_1)}$ for $0<t_0<t_1<T$. If $F'(T)<0$, then there exists $T_1>T$ such that $F(T_1)<\varepsilon$ and $T_1-T<t_0$. Then, by taking $u = u^* \mathds{1}_{(t_0-T_1+T,t_1-T_1+T)}$, we obtain a control on $(0,T)$ such that $J(u) = J(u^*)$ and $F(T)<\varepsilon$. However,this is not possible (see Lemma \ref{lemma:satur}). Thus, $F'(T)\geqslant 0$.

  Finally, let us show the claimed stationarity property of optimal values. To this aim, let us consider $T_1>T_2>T^*$. Since $T\mapsto J_{T,\bar{U},\varepsilon}$ is non-increasing, then $J_{T_1,\bar{U},\varepsilon}^{(\mathcal{S}_1)} \leqslant J_{T_2,\bar{U},\varepsilon}^{(\mathcal{S}_1)}$.
  Let us show that $J_{T_1,\bar{U},\varepsilon}^{(\mathcal{S}_1)} = J_{T_2,\bar{U},\varepsilon}^{(\mathcal{S}_1)}$. By contradiction, assume that $J_{T_1,\bar{U},\varepsilon}^{(\mathcal{S}_1)} < J_{T_2,\bar{U},\varepsilon}^{(\mathcal{S}_1)}$.
  Let us denote by $u_1^*$ (resp. $u_2^*$), the optimal solution on $(0,T_1)$ (resp. $(0,T_2)$). Then, from the results above, there exist $t_0^{(1)}<t_1^{(1)}$ and $t_0^{(2)}<t_1^{(2)}$ such that $u_1^* = u^*(\cdot-t_0^{(1)}) \mathds{1}_{(t_0^{(1)},t_1^{(1)})}$ and $u_2^* = u^*(\cdot-t_0^{(2)}) \mathds{1}_{(t_0^{(2)},t_1^{(2)})}$ with the same expression of $u^*$ given in \eqref{3eq:edo u}. From
  $$
   \int_{t_0^{(1)}}^{t_1^{(1)}} u^*(t-t_0^{(1)})\,dt =J_{T_1,\bar{U},\varepsilon}^{(\mathcal{S}_1)}
   < J_{T_2,\bar{U},\varepsilon}^{(\mathcal{S}_1)}=\int_{t_0^{(2)}}^{t_1^{(2)}} u^*(t-t_0^{(2)})\,dt,
  $$
  we deduce that $t_1^{(1)}-t_0^{(1)} < t_1^{(2)}-t_0^{(2)}$.
  Notice that, denoting $F_1$, resp. $F_2$, the solution to \eqref{eq:primal2} with $u=u_1^*$, resp. $u=u_2^*$, we have $F_1(t+t_0^{(1)}) = F_2(t+t_0^{(2)})$ for each $t\in(0,t_1^{(1)}-t_0^{(1)})$. Moreover, from the above remark, we have that $F_1'(T_1)\geqslant 0$, $F_2'(T_2)\geqslant 0$, and (using Lemma \ref{lem:u=0}) that $F_2'(t) > 0$ for $t> T_2$.
  Then, $F_1(t_1^{(1)}) = F_2(t_1^{(1)}+t_0^{(2)}-t_0^{(1)}) > F_2(t_1^{(2)})$.
  Since $t\mapsto F_1(t+t_1^{(1)})$ and $t\mapsto F_2(t+t_1^{(2)})$ verify the same dynamical system on $(0,T_1-t_1^{(1)})$, we deduce that $F_1(t+t_1^{(1)}) > F_2(t+t_1^{(2)})$ which implies in particular that $F_1(T_1) > F_2(T_1+t_1^{(2)}-t_1^{(1)}) \geqslant \varepsilon$. This is in contradiction with the fact that $F_1(T_1)=\varepsilon$ (see Lemma \ref{lemma:satur}).

  It remains to show that $F$ has its minimal value at $T$ and that $F'(T)=0$. Let us extend the definition of $F$ to $\RR_+$ by setting $u^*=0$ in $(T,+\infty)$. We already know that $F$ is non-increasing on $[0,T]$ according to Lemma~\ref{lem:f1neg} and therefore, $F'(T)\leq 0$, which leads to the conclusion since $F'(T)\geqslant 0$
  (see \eqref{eq:F(T)}).

\medskip


\noindent \textbf{Step 2: uniqueness property.} 
To conclude the proof, let us prove the uniqueness of the optimal control if $T>T^*$. We will use a constructive argument which is also used to derive an algorithm for solving numerically this problem in Section~\ref{sec:algo}.

For $(t_0,t_1)\in [0,T]^2$ such that $t_0\leqslant t_1$, introduce $(F_{(t_0,t_1)},{M_s}_{(t_0,t_1)})$ as the solution of the Cauchy problem
$$
\left\{\begin{array}{l}
\displaystyle\frac{d}{dt}\begin{pmatrix}F\\ M_s\end{pmatrix}  = \begin{pmatrix} f(F,M_s)\\ u - \delta_s M_s \end{pmatrix} \text{\quad in }[0,+\infty)
\\
\text{with }
u_{(t_0,t_1)} =\frac{\frac{\partial f}{\partial M_s}(F,M_s)\frac{\partial f}{\partial F}(F,M_s)+ \frac{\partial^2 f}{\partial M_s^2}(F,M_s) \delta_s M_s -\frac{\partial^2 f}{\partial M_s\partial F}(F,M_s)f(F,M_s) }{\frac{\partial^2 f}{\partial M_s^2}(F,M_s)}\mathds{1}_{[t_0,t_1]}  ,
\end{array}\right.
$$
complemented with the initial conditions $F(0) = \bar{F}$ and $M_s(0)=0$.

Let us first highlight that, with the notations of Theorem~\ref{theo:opt 2}, it is enough to consider the case where $t_0$ is equal to 0.
\begin{lemma}\label{lemma:ms0914}
Under the assumptions of Theorem~\ref{theo:opt 2} and if $T>T^*$, $u_{(t_0,t_1)}$ solves Problem~~\eqref{eq:opt2} if and only if the function $\tilde u$ given by $\tilde u({t})=\left. u_{(t_0,t_1)}\right|_{(t_0,T)}({t} -t_0)$ solves Problem~$(\mathcal{P}_{T-t_0,\overline{U},\varepsilon}^{(\mathcal{S}_1)})$.
\end{lemma}
Let us prove this lemma. The characterization of optimal controls in the previous step shows that any solution of Problem~~\eqref{eq:opt2} is of the form $u_{(t_0,t_1)}$  for some $0\leq t_0<t_1\leq T$. Observe that 
$$
\int_0^T u_{(t_0,t_1)}(t)\, dt=\int_0^{T-t_0}\tilde u(t)\, dt
$$
and furthermore, denoting respectively by $F_{u_{(t_0,t_1)}}$ and $F_{\tilde u}$ the trajectories associated to $u_{(t_0,t_1)}$ and $\tilde u$, one has by construction $F_{u_{(t_0,t_1)}}(T)=F_{\tilde u}(T-t_0)=\varepsilon$.
Since $T\in (T^*,+\infty)\mapsto J_{T,\bar{U},\varepsilon}^{(\mathcal{S}_1)}$ is nondecreasing, it is constant on $(T-t_0,T)$ and it follows that $\tilde u$ necessarily solves Problem~$(\mathcal{P}_{T-t_0,\overline{U},\varepsilon}^{(\mathcal{S}_1)})$. The proof of converse sense is exactly similar (and left to the reader).
This ends the proof of Lemma~\ref{lemma:ms0914}.

Let us argue by contradiction to prove the uniqueness of optimizers, assuming that Problem~\eqref{eq:opt2} has two solutions $u_{(t_0,t_1)}$ and $u_{(t_0',t_1')}$.
According to Lemma~\ref{lemma:ms0914} above, $u_{(0,t_1-t_0)}$ and $u_{(0,t_1'-t_0')}$
solve Problems~$(\mathcal{P}_{T-t_0,\overline{U},\varepsilon}^{(\mathcal{S}_1)})$ and $(\mathcal{P}_{T-t_0',\overline{U},\varepsilon}^{(\mathcal{S}_1)})$
respectively.
Without loss of generality, assume that $t_1'-t_0'\leqslant t_1-t_0$.
Since these controls are non-negative and $\int_{(0,T)}u_{(0,t_1-t_0)}=\int_{(0,T)}u_{(0,t_1'-t_0')}$, we infer that the function
$$
\frac{\frac{\partial f}{\partial M_s}(F,M_s)\frac{\partial f}{\partial F}(F,M_s)+ \frac{\partial^2 f}{\partial M_s^2}(F,M_s) \delta_s M_s -\frac{\partial^2 f}{\partial M_s\partial F}(F,M_s)f(F,M_s) }{\frac{\partial^2 f}{\partial M_s^2}(F,M_s)}
$$
vanishes on $(t_1-t_0,t_1'-t_0')$. 
If $t_1'-t_0'< t_1-t_0$, a standard regularity argument for Cauchy problems yields that {$u_{(0,t_1-t_0)}$} is real analytic on $(0,t_1-t_0)$, and vanishes hence identically on $(0,T)$. Therefore, $F_{u_{(0,t_1-t_0)}}(\cdot)=\bar F$ on $\RR_+$, which is impossible since $F_{u_{(0,t_1-t_0)}}(T-t_0)=\varepsilon$. 
Thus $t_1'-t_0'= t_1-t_0$ and $F_{u_{(0,t_1-t_0)}}=F_{u_{(0,t_1'-t_0')}}$ on $\RR_+$.
The function $F_{u_{(0,t_1-t_0)}}$ is $C^1$ on $\RR_+$ and one has $\lim_{t\to +\infty}F_{u_{(0,t_1-t_0)}}(t)=\bar F$, according to Proposition~\ref{prop:2 eq steady} and since the persistence equilibrium is the only one not being unstable. 
Let us denote by $t_2$ the time at which $F_{u_{(0,t_1-t_0)}}$ admits its first local minimum. Thanks to Lemma~\ref{lem:u=0}, $F_{u_{(0,t_1-t_0)}}$ decreases on $(0,t_2)$ and increases on $(t_2,+\infty)$. 
Since {$u_{(0,t_1-t_0)}$} solves Problems~$(\mathcal{P}_{T-t_0,\overline{U},\varepsilon}^{(\mathcal{S}_1)})$ and $(\mathcal{P}_{T-t_0',\overline{U},\varepsilon}^{(\mathcal{S}_1)})$, one has
$F_{u_{(0,t_1-t_0)}}'(T-t_0)=F_{u_{(0,t_1-t_0)}}'(T-t_0')=F'(t_2)=0$.
We deduce that $T-t_0=T-t_0'=t_2$, which shows the uniqueness of $t_0$ and $t_1$.
\end{proof}

\subsection{Proof of Property~\ref{prop:num}}
{We recall that, for each $\tau_1[0,T)$, 
$u_{\tau_1}$ and $F_{\tau_1}$, denote the solution to System \eqref{eq:primal3}, in particular $u_{\infty}$ and $F_{\infty}$ are the solutions to \eqref{eq:primal3} for $\tau_1=\infty$. 
}
Notice first that 
{$F^{\tau_1}$ is $C^1$ on $(0,+\infty)$ for each $\tau_1\in[0,T)$.}  
Denote by $t_{\mbox{\tiny max}}\in (0,+\infty]$ the minimal time $t$ at which it holds either  $(F^{\infty})'(t)> 0$ or $u_{\infty}(t)< 0$ {(we do not know a priori if $u_{\infty}$ can be negative or not since $u_{\infty}$ is the solution to System~\eqref{eq:primal3}, not to the optimal control problem)}. Its existence results from the form of optimal controls (see Theorem~\ref{theo:opt 2}).
Let us first prove that $t_{\mbox{\tiny max}}=+\infty$. 
Assume by contradiction that $t_{\mbox{\tiny max}}<+\infty$.
According to Proposition~\ref{prop:2 eq steady} and since the persistence equilibrium is the only one being not unstable, one has $\lim_{t\to +\infty}F^{t_{\mbox{\tiny max}}}(t)=\bar F$. 
Let $\tau_2$ be the time at which $F^{t_{\mbox{\tiny max}}}$ admits its first local minimum.
According to Lemma~\ref{lem:u=0}, $F^{t_{\mbox{\tiny max}}}$ is decreasing on $(0,\tau_2)$ and increasing on $(\tau_2,\infty)$. Hence, there exists $\delta>0$ such that $F^{t_{\mbox{\tiny max}}}\geqslant \delta$.
To reach a contradiction, let us consider a particular choice of $\varepsilon$, such that $\varepsilon\in (0,\delta)$. Let $T>0$ be large enough so that Problem~$(\mathcal{P}_{T,\overline{U},\varepsilon}^{(\mathcal{S}_1)})$ is well-posed (see Theorem~\ref{theo:opt 2}) and let $u_{(t_0,t_1)}$ be its unique solution.
According to Lemma~\ref{lemma:ms0914}, $u_{t_1-t_0}$
solves Problem~$(\mathcal{P}_{T-t_0,\overline{U},\varepsilon}^{(\mathcal{S}_1)})$.
Since $u_{t_1-t_0}$ is positive and $(F^{t_1-t_0})'> 0$ on $(0,t_1-t_0)$,
one has $t_1-t_0\leqslant t_{\mbox{\tiny max}}$.
By Lemma~\ref{monotonu2}, one has $F^{t_1-t_0}\geqslant F^{t_{\mbox{\tiny max}}}>\delta$ on $\RR_+$,
 which is in contradiction with the fact $F^{t_1-t_0}(T)=\varepsilon$. Thus $t_{\mbox{\tiny max}}=\infty$.

Item (i) can be obtained with the same reasoning as above and the fact that $t_{\mbox{\tiny max}}=+\infty$.
 Property (ii) is a consequence of Lemma~\ref{monotonu2} and again of the equality $t_{\mbox{\tiny max}}=+\infty$. 
Finally, the proof of the last claim (iii), establishing connections between Problem~\eqref{eq:opt2} under its general form and the control $u_{\tau_1}$, follows from Lemma~\ref{lemma:ms0914}, used to get the uniqueness of optimal controls in the proof of Theorem~\ref{theo:opt 2}.

\section{Comments on the optimal control problem and modeling issues}
\label{sec:otherpb}

In this section, we introduce two other possible choices of functionals to minimize and compare it to the one analyzed in the previous sections.
\subsection{\texorpdfstring{$L^2$} ~ functional}

Consider the optimal control problem
\begin{equation}\label{eq:opt3}
\underset{u\in \mathcal{U}_{T,\overline{U},\varepsilon}^{(\mathcal{S}_1)}}{\mbox{inf}}\tilde J(u),
\tag{\mbox{$\tilde{\mathcal{P}}_{T,\overline{U},\varepsilon}^{(\mathcal{S}_1)}$}}
\end{equation}
where the functional $\tilde J$ stands for the square of the $L^2$-norm of released mosquitoes over the horizon of time $T$, namely
\begin{equation*}
\tilde J(u):=\int_0^Tu(t)^2dt
\end{equation*} 
and 
$\mathcal{U}_{T,\overline{U},\varepsilon}^{(\mathcal{S}_1)}$ is defined by \eqref{calU2}.

\begin{theorem}\label{theo: L2}
Let $\eps\in(0,\bar{F})$.
For any $\overline{U} > U^*$ (defined by \eqref{eq:U*}), there exists a minimal time $\overline{T}(\bar U)>0$ such that for all $T\geqslant \overline{T}(\bar U)$,
the set ${\mathcal U}_{T,\overline{U},\varepsilon}^{(\mathcal{S}_1)}$ is nonempty and the optimal control
problem \eqref{eq:opt3} has a solution $u^*$. 
Moreover, for $\bar U>U^*$ and $T>\bar T(\bar U)$ and  large enough, there exists $\mu>0$ and $t_0\in [0,T)$ such that 
$$
u^*=\left\{\begin{array}{ll}
\bar U & \text{on }[0,t_0)\\
-\mu R, & \text{ on } [t_0,T]
\end{array}\right.
$$ 
where $R$ solves the dual system \eqref{eq:dual3}. Moreover $u^*(T)=0$ 
and the mappings $T\in [\bar T,+\infty)\mapsto \tilde  J_{T,\overline{U},C}^{(\mathcal{S}_1)}$ and $ \bar U\in (U^*,+\infty)\mapsto \tilde J_{T,\overline{U},C}^{(\mathcal{S}_1)}$ are non-increasing.
\end{theorem}


\begin{remark}
Similarly to the $L^1$ case, we have reduced the infinite dimensional  control problem to finite dimensional one, and  we therefore only need to determine $\mu$ and the value of $Q$ and $R$ at time $0$ to see $(F,M_s,Q,R)$ as the solution of a well-posed Cauchy problem.
\end{remark}
To prove Theorem \ref{theo: L2}, we first need the equivalent of Lemma \ref{lemma:cond_opt 2eq}.
\begin{lemma}\label{lemma:L2}
Consider  $u^*$ a solution to Problem~\eqref{eq:opt3}.
Then there exists $\lambda> 0$ such that
\begin{equation*}
\left\{\begin{array}{l}
\text{a.e. on }\{u^*=0\}\text{, one has } u^*(t)+\lambda R(t)\geqslant 0,\\
\text{a.e. on }\{0<u^*<\overline{U}\}\text{, one has } u^*(t)+\lambda R(t)=0,\\
\text{a.e. on }\{u^*=\overline{U}\}\text{, one has } u^*(t)+\lambda R(t)\leqslant 0.
\end{array}\right.
\end{equation*}
\end{lemma}
The proof is similar to Lemma \ref{lemma:cond_opt 2eq} and will be omitted.

\begin{proof}[Proof of Theorem \ref{theo: L2}]
  From \eqref{eq:dual3}, we deduce that
  $$
  R' = -\frac{\partial f}{\partial M_s} Q + \delta_s R > \delta_s R,
  $$
  where we use the fact that $Q>0$ and $\partial f/\partial M_s<0$. Since one has $R(T)=0$, we infer that $R<0$ in $[0,T)$ by using a standard Gronwall argument {(see e.g. \cite{gronwall1919note})}.
According to Lemma~\ref{lemma:L2}, one has $\{u^*=0\}=\emptyset$.

According to Lemma~\ref{lemma:Q,R}, since $0\leqslant \partial f/\partial M_s(M_s,F)\leqslant \operatorname{O}(|F|)$, one infers that $R$ is uniformly bounded on $[0,T]$, by a constant that does not depend on $\bar U$ and $T$. Then, by adapting the proof of Lemma~\ref{lemma:bar U 2 eq}, one shows that for $\overline{U}$ large enough, on any open interval of the open set $\{\bar U+\lambda R<0\}$, any local extremum of $R$ is a maximum.
It follows that $\{\bar U+\lambda R<0\}$ has at most one connected component of the form $(0,t_0)$, which leads to the conclusion.
\end{proof}

We provide on Figure~\ref{fig:steril L2}  the solutions of Problem~\eqref{eq:opt3} with  $T=200$, $\overline{U}=4000$, $\nu_E=0.05$ and $\varepsilon=\bar F/4$.
 We recover the theoretical result above, namely that the optimal control $u^*$ is positive, continuous and $u^*(T)=0$.

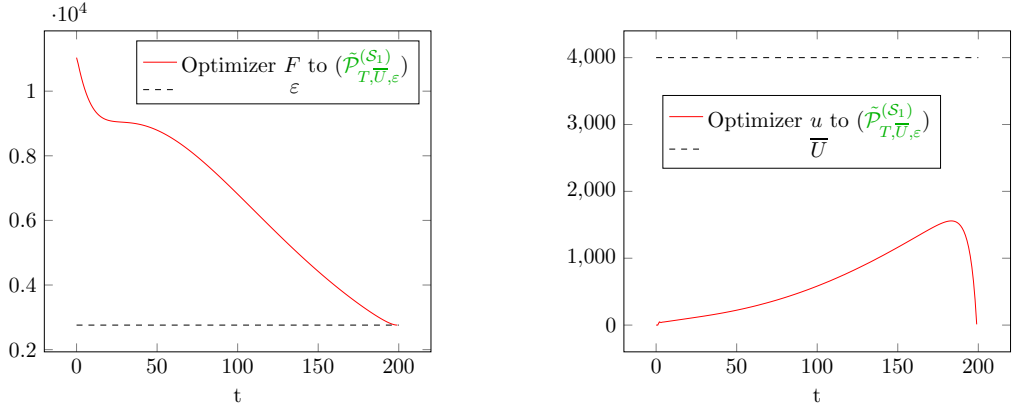
\begin{figure}[H]
\begin{center}~\hfill
\begin{tikzpicture}[thick,scale=0.75, every node/.style={scale=1.0}] \begin{axis}[xlabel=t,
,legend pos=north east, legend columns=1]
 \addplot[color=red]coordinates { 
(0.0,11037.970588)
(1.0,10846.685132)
(2.0,10624.359076)
(3.0,10422.523841)
(4.0,10240.213967)
(5.0,10076.399584)
(6.0,9930.007328)
(7.0,9799.9257675)
(8.0,9685.015907)
(9.0,9584.1265144)
(10.0,9496.1084184)
(11.0,9419.824728)
(12.0,9354.1636291)
(13.0,9298.0490432)
(14.0,9250.4536441)
(15.0,9210.396167)
(16.0,9176.9617599)
(17.0,9149.291716)
(18.0,9126.5904219)
(19.0,9108.1274663)
(20.0,9093.2364274)
(21.0,9081.3139342)
(22.0,9071.8172587)
(23.0,9064.2583509)
(24.0,9058.2044062)
(25.0,9053.2726697)
(26.0,9049.1261121)
(27.0,9045.4694459)
(28.0,9042.0454575)
(29.0,9038.6321094)
(30.0,9035.0412122)
(31.0,9031.1084234)
(32.0,9026.6910272)
(33.0,9021.6678926)
(34.0,9015.9378234)
(35.0,9009.4167541)
(36.0,9002.0349623)
(37.0,8993.7344613)
(38.0,8984.4679145)
(39.0,8974.2012903)
(40.0,8962.9079274)
(41.0,8950.5681269)
(42.0,8937.1683458)
(43.0,8922.7003745)
(44.0,8907.1595044)
(45.0,8890.5467452)
(46.0,8872.8666465)
(47.0,8854.1273293)
(48.0,8834.3386261)
(49.0,8813.5103771)
(50.0,8791.6584198)
(51.0,8768.7980161)
(52.0,8744.9443594)
(53.0,8720.11502)
(54.0,8694.3278467)
(55.0,8667.6027102)
(56.0,8639.9601164)
(57.0,8611.4183958)
(58.0,8581.9938963)
(59.0,8551.7059624)
(60.0,8520.5778256)
(61.0,8488.6325872)
(62.0,8455.8921929)
(63.0,8422.3767699)
(64.0,8388.1052922)
(65.0,8353.0952304)
(66.0,8317.3563577)
(67.0,8280.9091841)
(68.0,8243.7781277)
(69.0,8205.9937862)
(70.0,8167.5783252)
(71.0,8128.5516144)
(72.0,8088.9331338)
(73.0,8048.7434651)
(74.0,8008.0019989)
(75.0,7966.7294806)
(76.0,7924.9447503)
(77.0,7882.6654442)
(78.0,7839.907586)
(79.0,7796.6925086)
(80.0,7753.0369656)
(81.0,7708.9585864)
(82.0,7664.4750376)
(83.0,7619.6032644)
(84.0,7574.3600503)
(85.0,7528.7601083)
(86.0,7482.8203718)
(87.0,7436.5587883)
(88.0,7389.9922107)
(89.0,7343.136518)
(90.0,7296.0080032)
(91.0,7248.6216012)
(92.0,7200.9886967)
(93.0,7153.123025)
(94.0,7105.0426887)
(95.0,7056.764907)
(96.0,7008.3032424)
(97.0,6959.6694545)
(98.0,6910.8782627)
(99.0,6861.9424468)
(100.0,6812.8764442)
(101.0,6763.6936152)
(102.0,6714.4062182)
(103.0,6665.0255542)
(104.0,6615.5639133)
(105.0,6566.0339339)
(106.0,6516.4476736)
(107.0,6466.8171408)
(108.0,6417.1586374)
(109.0,6367.48644)
(110.0,6317.8050645)
(111.0,6268.1224221)
(112.0,6218.4493366)
(113.0,6168.7965921)
(114.0,6119.1746119)
(115.0,6069.5933358)
(116.0,6020.0625247)
(117.0,5970.5927789)
(118.0,5921.2009827)
(119.0,5871.8891421)
(120.0,5822.6636812)
(121.0,5773.5327495)
(122.0,5724.5047123)
(123.0,5675.5883925)
(124.0,5626.7918842)
(125.0,5578.1219731)
(126.0,5529.5856958)
(127.0,5481.1891115)
(128.0,5432.9389189)
(129.0,5384.8425447)
(130.0,5336.9073711)
(131.0,5289.1408371)
(132.0,5241.5502299)
(133.0,5194.1449445)
(134.0,5146.9318921)
(135.0,5099.9144181)
(136.0,5053.0969054)
(137.0,5006.4832663)
(138.0,4960.074299)
(139.0,4913.8749336)
(140.0,4867.8917619)
(141.0,4822.1323852)
(142.0,4776.6020035)
(143.0,4731.3041537)
(144.0,4686.2401799)
(145.0,4641.4137449)
(146.0,4596.8288646)
(147.0,4552.4893459)
(148.0,4508.3991434)
(149.0,4464.5625331)
(150.0,4420.9829494)
(151.0,4377.6634212)
(152.0,4334.6069768)
(153.0,4291.816543)
(154.0,4249.2952829)
(155.0,4207.0448825)
(156.0,4165.0678061)
(157.0,4123.366435)
(158.0,4081.9431306)
(159.0,4040.8005572)
(160.0,3999.9412896)
(161.0,3959.3675188)
(162.0,3919.0822612)
(163.0,3879.0863143)
(164.0,3839.3821319)
(165.0,3799.9712651)
(166.0,3760.8541235)
(167.0,3722.033356)
(168.0,3683.5125012)
(169.0,3645.2947445)
(170.0,3607.382912)
(171.0,3569.7814019)
(172.0,3532.4945953)
(173.0,3495.5276735)
(174.0,3458.886538)
(175.0,3422.5776583)
(176.0,3386.6087674)
(177.0,3350.9891037)
(178.0,3315.7300057)
(179.0,3280.8445223)
(180.0,3246.3476896)
(181.0,3212.2582336)
(182.0,3178.5983447)
(183.0,3145.3950942)
(184.0,3112.6807367)
(185.0,3080.4947114)
(186.0,3048.8847096)
(187.0,3017.9090936)
(188.0,2987.6396465)
(189.0,2958.1656723)
(190.0,2929.5975402)
(191.0,2902.0734112)
(192.0,2875.7680678)
(193.0,2850.908435)
(194.0,2827.7904761)
(195.0,2806.8099243)
(196.0,2788.5094893)
(197.0,2773.656268)
(198.0,2763.3890805)
(199.0,2759.4930142)

 };
 
   \addplot[dashed,color=black]coordinates { 
  (0,2759.492647058823)
  (200,2759.492647058823)
 }; 
 
 \legend{Optimizer $F$ to \eqref{eq:opt3},$\varepsilon$}
\end{axis} 
\end{tikzpicture}
\hfill
\begin{tikzpicture}[thick,scale=0.75, every node/.style={scale=1.0}] \begin{axis}[xlabel=t,
legend style={at={(0.82,0.8)}}
]
 \addplot[color=red]coordinates { 
(0.0,0.0)
(1.0,0.0)
(2.0,47.207123427)
(3.0,38.384883769)
(4.0,44.629604568)
(5.0,47.011776543)
(6.0,50.520323795)
(7.0,53.941107198)
(8.0,57.439811417)
(9.0,60.868451356)
(10.0,64.380356636)
(11.0,67.868684078)
(12.0,71.437211621)
(13.0,74.770124808)
(14.0,78.470932919)
(15.0,81.757383562)
(16.0,85.315938553)
(17.0,88.840975845)
(18.0,92.363876487)
(19.0,95.899964753)
(20.0,99.43724456)
(21.0,102.94582493)
(22.0,106.55328648)
(23.0,110.10914912)
(24.0,113.68807914)
(25.0,117.2842308)
(26.0,120.90692204)
(27.0,124.55051213)
(28.0,128.22923389)
(29.0,131.826312)
(30.0,135.55474845)
(31.0,139.37228138)
(32.0,143.23358501)
(33.0,147.12639084)
(34.0,151.06423225)
(35.0,155.06253354)
(36.0,159.13627322)
(37.0,163.31731742)
(38.0,167.44234149)
(39.0,171.68466476)
(40.0,175.98501514)
(41.0,180.36156196)
(42.0,184.7938308)
(43.0,189.35106163)
(44.0,193.91807382)
(45.0,198.57398505)
(46.0,203.27868261)
(47.0,208.05413944)
(48.0,213.01463183)
(49.0,217.87977611)
(50.0,222.90725022)
(51.0,228.04790381)
(52.0,233.20918521)
(53.0,238.49375086)
(54.0,243.80802482)
(55.0,249.18967845)
(56.0,254.71459492)
(57.0,260.41275574)
(58.0,266.12803036)
(59.0,271.81317726)
(60.0,277.59769582)
(61.0,283.47995304)
(62.0,289.48944574)
(63.0,295.635546)
(64.0,301.80530109)
(65.0,308.51249136)
(66.0,314.83920151)
(67.0,321.30853319)
(68.0,327.4792921)
(69.0,334.0975239)
(70.0,340.77943979)
(71.0,347.58476188)
(72.0,354.39479728)
(73.0,361.3762491)
(74.0,368.33291604)
(75.0,375.47206913)
(76.0,382.67874758)
(77.0,390.14170081)
(78.0,397.34369682)
(79.0,404.92641559)
(80.0,412.48536065)
(81.0,420.14163007)
(82.0,427.92070243)
(83.0,435.74648193)
(84.0,443.80244397)
(85.0,451.82212793)
(86.0,459.88769115)
(87.0,468.07644548)
(88.0,476.39385648)
(89.0,484.75519894)
(90.0,493.21993877)
(91.0,501.99156006)
(92.0,510.75728057)
(93.0,519.3779811)
(94.0,528.11498137)
(95.0,537.07818373)
(96.0,546.27138511)
(97.0,555.29704341)
(98.0,564.62144183)
(99.0,573.85441053)
(100.0,583.26059351)
(101.0,592.77429863)
(102.0,602.45264994)
(103.0,612.14344309)
(104.0,621.90421767)
(105.0,631.76170599)
(106.0,641.78207842)
(107.0,651.55057926)
(108.0,661.35432579)
(109.0,672.03185455)
(110.0,682.45417951)
(111.0,692.83931817)
(112.0,703.32774753)
(113.0,713.89730932)
(114.0,724.59010467)
(115.0,735.30860024)
(116.0,746.3071181)
(117.0,756.33762884)
(118.0,768.0689493)
(119.0,779.18348432)
(120.0,790.41935699)
(121.0,801.71526795)
(122.0,813.03181184)
(123.0,824.4535745)
(124.0,836.08362201)
(125.0,847.70403939)
(126.0,859.53803489)
(127.0,871.37831939)
(128.0,883.22583331)
(129.0,895.161538)
(130.0,907.11202165)
(131.0,919.22247788)
(132.0,931.05614372)
(133.0,943.14670507)
(134.0,955.69201884)
(135.0,968.13301577)
(136.0,980.60667004)
(137.0,993.71545061)
(138.0,1006.3225664)
(139.0,1018.9717868)
(140.0,1031.4233894)
(141.0,1044.2314067)
(142.0,1057.121099)
(143.0,1070.4775515)
(144.0,1083.5159822)
(145.0,1096.6451403)
(146.0,1109.8404678)
(147.0,1123.0751261)
(148.0,1136.2634186)
(149.0,1149.6126788)
(150.0,1163.0381862)
(151.0,1176.4813774)
(152.0,1190.0171759)
(153.0,1203.4470565)
(154.0,1217.2356101)
(155.0,1230.8526486)
(156.0,1244.5617476)
(157.0,1258.2983255)
(158.0,1272.0052412)
(159.0,1285.7038925)
(160.0,1299.5758282)
(161.0,1313.0938581)
(162.0,1327.2396978)
(163.0,1340.8924056)
(164.0,1354.7439272)
(165.0,1368.9159684)
(166.0,1382.6014992)
(167.0,1396.1543987)
(168.0,1409.6787613)
(169.0,1423.3096052)
(170.0,1436.4352466)
(171.0,1449.5753674)
(172.0,1462.3806429)
(173.0,1474.8865128)
(174.0,1487.1211161)
(175.0,1498.8544866)
(176.0,1510.0537145)
(177.0,1520.4499142)
(178.0,1530.0750333)
(179.0,1538.8767326)
(180.0,1546.2100976)
(181.0,1552.2119333)
(182.0,1556.2128792)
(183.0,1558.174475)
(184.0,1557.2669416)
(185.0,1553.1712179)
(186.0,1544.9770811)
(187.0,1531.8926701)
(188.0,1512.4913659)
(189.0,1485.8769007)
(190.0,1450.0655233)
(191.0,1403.3329504)
(192.0,1341.8666702)
(193.0,1263.1959776)
(194.0,1162.2420458)
(195.0,1032.8605949)
(196.0,867.92738209)
(197.0,654.41725342)
(198.0,377.1819081)
(199.0,11.254411888)
 };
 
  \addplot[dashed,color=black]coordinates { 
  (0,4000)
  (200,4000)
 }; 
 \legend{Optimizer $u$ to \eqref{eq:opt3}, $\bar U$}
\end{axis} 
\end{tikzpicture}\hfill~
\end{center}
\caption{Solution of the optimal control problems \eqref{eq:opt3} 
with  $T=200$, $\overline{U}=4000$, $\nu_E=0.05$ and $\varepsilon=\bar F/4$.}
\label{fig:steril L2}
\end{figure}

\begin{remark}
In optimal control theory, the $L^2$-norm  is often preferred to the $L^1$-norm for differentiability issues. However, from a biological point of view, the $L^1$-norm is more relevant since it stands for the amount of individuals.
Moreover, as it can be seen on Figure~\ref{fig:steril1}, the optimal control for the $L^1$-norm is sparse unlike the one for the $L^2$-norm, which is interesting from a practical point of view. 
\end{remark}

\subsection{Dual optimal  control problem}

Consider the optimal control problem
\begin{equation}\label{eq:opt4}
\underset{u\in \mathcal{U}_{T,\overline{U},C}^{(\mathcal{S}_1)}}{\mbox{inf}}\hat J(u),
\tag{\mbox{$\hat{\mathcal{P}}_{T,\overline{U},C}^{(\mathcal{S}_1)}$}}
\end{equation}
where the functional $\hat J$ stands for the total number of eggs and females (with some weights) at time $T$, namely
\begin{equation*}
\hat J(u):=F(T),
\end{equation*} 
where $F$ solves System~\eqref{eq:primal2} associated to the control $u$ and 
$\mathcal{U}_{T,\overline{U},C}^{(\mathcal{S}_1)}$ is the set of admissible controls, chosen so that:
\begin{itemize}
\item the rate of sterile male mosquito release is non-negative, uniformly bounded by a positive constant $\overline{U}$;
\item the total number of released sterilized males over the time interval $(0,T)$ is assumed to be lower than $C$.
\end{itemize}
Hence, $\mathcal{U}_{T,\overline{U},C}^{(\mathcal{S}_1)}$ is defined by
\begin{equation*}
{\mathcal U}_{T,\bar{U},C}^{(\mathcal{S}_1)} := \Big\{ u\in L^\infty(0,T):0\leqslant u \leqslant \overline{U} \text{ a.e. in }(0,T),\int_0^T u(t)\,dt\leqslant C\Big\}.
\end{equation*}

In \cite{MBE}, a similar optimal control problem has been considered.
Problem~\eqref{eq:opt4} can be seen as a dual version of 
\eqref{eq:opt2} in the following sense: let  $u^*\in \mathcal{U}_{T,\overline{U},\varepsilon}^{(\mathcal{S}_1)}$ be  a solution to Problem~\eqref{eq:opt2} for some given $T$, $\overline{U}$ and $\varepsilon$. 
Then, the control $u^*$ is a solution to Problem~\eqref{eq:opt4} for the parameter choice
$C:=\int_0^Tu^*(t)\,dt$. 
Indeed, assume by contradiction that 
there exists $u \in  \mathcal{U}_{T,\overline{U},C}^{(\mathcal{S}_1)}$ such that 
$$\hat J(u)<\hat J(u^*),$$
that is 
$F(T)<F^{*}(T)=\varepsilon$,
where $(F,M_s)$ and $(F^*,M_s^*)$ are the solutions to system \eqref{eq:primal2} associated to $u$ and $u^*$ respectively.
By mimicking the argument provided in the proof of Lemma~\ref{lemma:satur 4 eq}, 
we reach a contradiction, which shows that  $u^*$ is a solution to Problem~\eqref{eq:opt4}.

Respectively, let  $\hat u^*\in \mathcal{U}_{T,\overline{U},C}^{(\mathcal{S}_1)}$ be  an optimizer to Problem~\eqref{eq:opt4} for some given $T$, $\overline{U}$ and $C$. 
By using the same argument, one shows that $\hat u^*$ is an optimizer to Problem~\eqref{eq:opt2} for the parameter choice $\varepsilon:=F^*(T)$.

\section{Conclusion}

  In this paper, we have determined the optimal release function which minimizes the number of sterilized males needed when performing the sterile insect technique (SIT) to reduce the size of a population of mosquitoes to a given value.
  Starting from a differential system modeling the dynamics of the mosquito population, we simplify it to obtain a reduced system, which is a good approximation, and for which we are able to compute precisely the optimal solution. These theoretical results are illustrated thanks to some numerical simulations. Notice that once the form of the theoretical solution is known, efficient algorithms may be designed to compute quickly the numerical solution.

  Obviously, when the final number of mosquitoes is fixed, there is a minimal time to perform the releases in order to reach this value. Interestingly, the number of sterilized males needed is non-increasing with respect to the time of the experiment, meaning that the longer the duration of the experiment, the lower the number of sterilized males. However, there is a maximal time above which the minimal number of sterilized males needed is stationary. In this case, the optimal release function is given by a singular arc sandwiched between two regions where it is zero. The knowledge of the existence of this time may be interesting for practical applications since for larger time the number of sterilized males stays constant.

  Thanks to our results, we are able to give a precise description of the temporal distribution of the releases to optimize given scenarios. A natural extension of this work is to add the spatial distribution of mosquitoes since it may have a big impact on the success of the SIT ({we refer the interested reader to \cite{Mosquitocemracs2018,CoCV}} for some simple spatial models). In particular, most experiments have been performed in isolated regions to avoid re-invasion from the outside. But even in isolated regions the question of knowing where to perform the releases to have the best efficiency of the SIT is still open.

\appendix
\begin{center}
\fbox{\Large{\bf Appendix}}
\end{center}
\section{Mathematical properties of the dynamical systems}\label{sec:mathpropdynsys}

\subsection*{Proof of Proposition \ref{prop:4 eq steady}}
Let us assume that $u(\cdot)=0$. Then, the equilibria $(\bar E,\bar M,\bar F,\bar M_s)$ of System~\eqref{eq:S1} solve
\begin{multline*}
0 = \beta_E \bar F \left(1-\frac{\bar E}{K}\right) - \big( \nu_E + \delta_E \big)\bar  E= (1-\nu)\nu_E\bar  E - \delta_M\bar  M \\
= \nu\nu_E\bar  E \frac{\bar M}{\bar M+\gamma_s\bar  M_s} - \delta_F\bar  F= - \delta_s\bar  M_s.
\end{multline*}
We thus infer that $\bar{M}_s=0$ and 
\begin{equation*}
0 = \beta_E \bar F \left(1-\frac{\bar E}{K}\right) - \big( \nu_E + \delta_E \big)\bar  E= (1-\nu)\nu_E\bar  E - \delta_M\bar  M= 
\nu\nu_E\bar  E  - \delta_F\bar  F.
\end{equation*}
Then, $(0,0,0,0)$ is an equilibrium, and the only non-zero equilibrium is 
\begin{equation*}
\bar E=K \left(1-\frac{\delta_F\big( \nu_E + \delta_E \big)}{\beta_E \nu\nu_E}\right),\quad 
\bar  M=\frac{(1-\nu)\nu_E}{\delta_M}\bar  E ,  \quad
\bar  F=\frac{\nu\nu_E}{\delta_F}\bar  E  , \quad \bar M_s=0
\end{equation*}
whence \eqref{eq:equi}.

 Let us show that $(0,0,0,0)$ is unstable. Using that  
 $M_s(t)=e^{-\delta_st}M_{s}(0)$ for $t\geqslant0$, we deduce that $M(t)\geqslant e^{-\delta_M t}M(0)$ for $t\geqslant0$ according to \eqref{eq:cond prop}, and it follows that for any $\epsilon>0$, there exists $t^*>0$ such that 
 $ \gamma_s M_s(t)<\epsilon M(t)$  for all $t\geqslant t^*$. By using a standard comparison principle, we get that for all $t\geqslant t^*$ 
 \begin{equation*}
(E(t), F(t))\geqslant (E_1(t),F_1(t)),
 \end{equation*}
 the inequality being understood component by component, where, $(E_1,F_1)$ solves
 \begin{equation}
  \left\{
    \begin{aligned}
&\frac{dE_1}{dt}  =  \beta_E F_1 \left(1-\frac{E_1}{K}\right) - \big( \nu_E + \delta_E \big) E_1, \quad t\geqslant t^* \\
&\frac{dF_1 }{dt} =  \frac{\nu\nu_E }{1+\epsilon}E_1 - \delta_F F_1
\end{aligned}
  \right.
  \label{eq:F1}
\end{equation}
complemented with the initial data $(E_1(t^*),F_1(t^*))=(E(t^*),F(t^*))$.
An easy computation yields that the Jacobian matrix of System~\eqref{eq:F1} at $(0,0)$ is
$$
\begin{pmatrix}
- \nu_E - \delta_E &\beta_E  \\
\frac{\nu\nu_E }{1+\epsilon}& - \delta_F 
\end{pmatrix},
$$
whose determinant expands as $\delta_F(\nu_E+\delta_E)-\nu\beta_E \nu_E +\operatorname{O}(\epsilon)$.
According to \eqref{eq:cond prop}, we infer that the Jacobian matrix has a positive root whenever $\epsilon$ is chosen small enough, which leads to the conclusion.

Let us now investigate the stability of $(\bar E,\bar M,\bar F,0)$ for System~\eqref{eq:S1}. Easy computations yield that the Jacobian matrix of System~\eqref{eq:S1} at $(\bar E,\bar M,\bar F,0)$ reads
\begin{multline*}
 \begin{pmatrix}
-\frac{\beta_E}{K}\bar F-(\nu_E+\delta_E) &0& \beta_E\left(1-\frac{\bar E}{K}\right)&0  \\
(1-\nu)\nu_E&-\delta_M&0&0\\
\nu\nu_E &0 & -\delta_F &-\frac{\gamma_s\nu\nu_E\bar E}{\bar M}\\
0&0&0&-\delta_s
\end{pmatrix}\\=  \begin{pmatrix}
-\frac{\nu\nu_E\beta_E}{\delta_F} &0&\frac{\delta_F(\nu_E+\delta_E)}{\nu\nu_E}&0  \\
(1-\nu)\nu_E&-\delta_M&0&0\\
\nu\nu_E &0 & -\delta_F &-\frac{\gamma_s\nu\delta_M}{1-\nu}\\
0&0&0&-\delta_s
\end{pmatrix}.
\end{multline*}
so that the four eigenvalues are $-\delta_s$, $-\delta_M$, and the two (complex conjugate) roots of the polynomial
$$
P=X^2+\left(\frac{\nu\nu_E\beta_E}{\delta_F}+\delta_F\right)X+(\nu\nu_E\beta_E-\delta_F(\nu_E+\delta_E)), 
$$
which have a negative real part under \eqref{eq:cond prop}.
It follows that $(\overline{E},\overline{M}, \overline{F},0)$ is {locally} asymptotically stable.

Let us now prove the second part of the proposition. We first notice that the set $[0,K]\times\RR_+^3$ is stable whenever $u$ is non-negative. We claim that System~\eqref{eq:S1} is monotone on this set (see Lemma~\ref{monotonu}). 
 If $u$ belongs to $L^\infty(0,T,\RR_+)$, then $(0,0,0,0)$ is obviously a subsolution of System~\eqref{eq:S1} whereas $(\bar{E},\bar{M},\bar{F},\|u\|_\infty/\delta_s)$ is a supersolution. A comparison argument allows us to conclude that $[0,\bar{E}]\times [0,\bar{M}]\times [0,\bar{F}]\times \RR_+$ is stable. Furthermore, if all initial data are positive, then so are the functions $E$, $M$, $F$ and $M_s$, and we get that 
 $$
 E(t)\geqslant e^{-(\nu_E+\delta_E)t} E(0), \quad M(t) \geqslant e^{-\delta_M t} M(0), \quad F(t)\geqslant e^{-\delta_F t} F(0)
 $$ 
 for all $t\geq 0$, implying that these quantities cannot vanish.

  Finally, let us consider the case where $u(\cdot)=\bar{U}$, where $\bar{U}>U^*$. In that case, the non-zero equilibria of System~\eqref{eq:S1} solve
  $$
  \bar{M_s} = \frac{\bar{U}}{\delta_S}, \quad
  \bar{M} = \frac{(1-\nu)\nu_E}{\delta_M} \bar{E}, \quad
  \bar{F} = \frac{\nu \nu_E}{\delta_F} \frac{\bar{E}\,\bar{M}}{\bar{M}+\gamma_s \bar{M}_s},
 $$and$$
  0 = \beta_E \bar{F}\left(1-\frac{\bar{E}}{K}\right) - (\nu_E+\delta_E)\bar{E}
  $$
plugging the three first above equalities into this latter equation, we get that $\bar{E}$ satisfies
$$
  \bar{E}= 0$$or$$
  \frac{\beta_E\nu(1-\nu)\nu_E^2}{\delta_F \delta_M K} \bar{E}^2 -
  \frac{\beta_E\nu(1-\nu)\nu_E^2}{\delta_F \delta_M}\left(1-\frac{\delta_F(\nu_E+\delta_E)}{\beta_E \nu \nu_E}\right) \bar{E} + \frac{\gamma_s (\nu_E+\delta_E)}{\delta_s}\bar{U}=0.
$$
One easily checks that this second order polynomial with unknown $\bar{E}$ does not have any real solutions if $\bar{U}>U^*$. In this case, the only equilibrium is the extinction equilibrium. We infer that any non-negative solution to the Cauchy problem converges to this unique steady state.
\qed

\subsection*{Proof of Proposition \ref{prop:2 eq steady}}
Let us assume that $u(\cdot)=\bar{U}$. The equilibria are obtained by solving the system
 \begin{equation*}
 0 = f(\bar{F},\bar{M}_s) = \bar{U} - \delta_s \bar{M}_s,
\end{equation*}
which is equivalent to $\bar{M}_s=\bar{U}/\delta_s$ and 
  \begin{multline*}
    \bar{F}\left(
    \nu(1-\nu)\beta_E^2\nu_E^2 \bar{F} - \delta_F \big(\frac{\beta_E \bar{F}}{K} + \nu_E+\delta_E \big)
    \Big((1-\nu)\nu_E\beta_E \bar{F} \right.\\\left.+ \delta_M \gamma_s \bar{M_s} \big(\frac{\beta_E \bar{F}}{K} + \nu_E+\delta_E \big) \Big)
  \right) = 0.
  \end{multline*}
It follows that if $\bar{U} = 0$, then there are exactly two different solutions for this equation, namely $\bar{F}=0$ or $\bar{F}$ given by \eqref{eq:equi}.
 A straightforward computation yields that if $\bar{U}>U^*$, then the equation above has no positive solution and furthermore, $f(F,\bar{M_s})<0$ for every $F>0$. Therefore, we infer that any non-negative solution converges to the steady state $(0,\bar{M_s})$ if $\bar{U}>U^*$.

Let us assume that $u(\cdot)=0$. Using that 
$$
M_s(t)=e^{-\delta_st}M_{s}(0), \quad F(t)\geqslant e^{-\delta_F t}F({0})
$$ 
for $t\geqslant0$ and the fact that $\delta_s>\delta_F$, we get that, for all $\eta>0$, there exists $t^*>0$ such that 
 $$
\delta_M \gamma_s (\frac{\beta_E F(t)}{K} + \nu_E+\delta_E) M_s(t)<\eta (1-\nu)\nu_E \beta_E F(t) \quad \mbox{ for all }t\geqslant t^*.
 $$
 It follows that
$$
  f(F(t),M_s(t)) >  \tilde f(F(t)):=\frac{\nu\beta_E \nu_E F(t)}{\big(\frac{\beta_E F(t)}{K} + \nu_E+\delta_E \big)  (1+\eta)} - \delta_F F(t),
$$
whenever 
 $t\geqslant t^*$. 
 We conclude by observing that 
$\tilde f'(0)>0$ for $\eta$ small enough, 
 so that we get the instability of the equilibrium $(0,0)$. 

Finally, let us investigate the stability of the persistence steady state $(\bar{F},0)$ of System~\eqref{eq:primal2}. One computes
\begin{equation*}
\begin{array}{l}
\frac{\partial f}{\partial F}(\bar F,0)\\ = \nu(1-\nu)\beta_E^2 \nu_E^2 \frac{2 \bar{F}\big(\frac{\beta_E \bar{F}}{K} + \nu_E+\delta_E \big)(1-\nu)\nu_E \beta_E \bar{F}
 - \bar{F}^2(1-\nu)\nu_E \beta_E\big(2\frac{\beta_E \bar{F}}{K} + \nu_E+\delta_E \big)}
 {\big(\frac{\beta_E \bar{F}}{K} + \nu_E+\delta_E \big)^2 \big((1-\nu)\nu_E \beta_E \bar{F}\big)^2} - \delta_F \\[4mm]
   =
   \delta_F\Big({\frac{1}{\mathcal{R}_0}}
  - 1\Big),
\end{array}\end{equation*}
which is negative under condition \eqref{eq:cond prop}. Because of the form of System~\eqref{eq:primal2}, it follows that $(\overline{F},0)$ is a {locally} asymptotically stable steady state.
 Finally, a standard comparison argument for Cauchy problems yields that if $0<F(0)< \bar{F}$ and $u(\cdot)\geqslant 0$ then, we have $0<F(t)<\bar{F}$ for all $t\geqslant 0$.
\qed


\section*{Acknowledgements}

The authors acknowledge
the support of  the program STIC AmSud (project 20-STIC-05)
and of the Project ``Analysis and simulation of optimal shapes - application to life sciences'' of the Paris City Hall.

%
%



\bibliographystyle{abbrv}
\bibliography{biblio}

\end{document}